\documentclass[11pt]{article}

\usepackage[margin=1truein]{geometry}

\usepackage{setspace}
\usepackage{lscape}
\usepackage{pdflscape}
\usepackage{microtype}

\usepackage{amsmath,amssymb,amsthm} 

\usepackage{wrapfig} 

\usepackage{mathrsfs} 
\usepackage{url} 
\usepackage{latexsym} 
\usepackage{bm} 

\usepackage{mathtools}

\usepackage{booktabs,multirow}
\usepackage{array}
\newcolumntype{P}[1]{>{\centering\arraybackslash}p{#1}}
\usepackage{diagbox}

\usepackage{algorithm}
\usepackage[noend]{algpseudocode}

\algnewcommand{\Break}{\textbf{break}}
\algnewcommand{\algorithmicgoto}{\textbf{go to} Line}
\algnewcommand{\Goto}[1]{\algorithmicgoto~\ref{#1}}
\algblockdefx{MRepeat}{EndRepeat}{\textbf{repeat}}{}
\algnotext{EndRepeat}

\usepackage[sort&compress,numbers]{natbib}

\usepackage{thmtools,thm-restate}

\usepackage{hyperref}
\usepackage{cleveref}
\expandafter\def\csname ver@etex.sty\endcsname{3000/12/31} 
\usepackage{autonum} 
\makeatletter
\autonum@generatePatchedReferenceCSL{Cref} 
\makeatother

\usepackage{tikz,tikz-3dplot}
\usetikzlibrary{positioning,shapes,shapes.callouts,shapes.multipart,arrows.meta,fit,calc}
\usepackage{pgfplots}
\pgfplotsset{compat=newest}
\usepackage{pxpgfmark} 
\usetikzlibrary{bayesnet} 

\usepackage{multicol}
\usepackage{wasysym} 
\usepackage{adjustbox}
\usepackage{xparse} 
\usepackage{pbox} 
\usepackage{authblk} 
\usepackage{xspace} 

\usepackage[subrefformat=parens]{subcaption}
\usepackage[inline]{enumitem} 
\hypersetup{
  colorlinks,
  linkcolor={green!35!black},
  citecolor={green!35!black},
  urlcolor={blue!50!black}
}

\DeclareMathOperator*{\argmin}{argmin}

\DeclareMathOperator{\EE}{\mathbb{E}}

\DeclarePairedDelimiter\floor{\lfloor}{\rfloor}
\DeclarePairedDelimiterX\inner[2]{\langle}{\rangle}{{#1},{#2}}

\DeclarePairedDelimiter\norm{\|}{\|}

\DeclarePairedDelimiter\set{\{}{\}}
\DeclarePairedDelimiter\prn{(}{)}
\DeclarePairedDelimiter\bra{[}{]}
\DeclarePairedDelimiterX\Set[2]{\{}{\}}{\mspace{2mu}{#1}\;\delimsize|\;{#2}\mspace{2mu}}
\DeclarePairedDelimiterX\Prn[2]{(}{)}{\mspace{2mu}{#1}\;\delimsize|\;{#2}\mspace{2mu}}
\DeclarePairedDelimiterX\Bra[2]{[}{]}{\mspace{2mu}{#1}\;\delimsize|\;{#2}\mspace{2mu}}

\newcommand{\R}{\mathbb R}

\newcommand{\0}{\mathbf 0}

\renewcommand{\epsilon}{\varepsilon}

\NewDocumentCommand{\exsub}{s m O{} m}{%
  \IfBooleanT{#1}{\EE_{#2}\nolimits\bra*{#4}}%
  \IfBooleanF{#1}{\EE_{#2}\nolimits\bra[#3]{#4}}%
}
\NewDocumentCommand{\ex}{s O{} m}{%
  \IfBooleanT{#1}{\EE\nolimits\bra*{#3}}%
  \IfBooleanF{#1}{\EE\nolimits\bra[#2]{#3}}%
}
\NewDocumentCommand{\cex}{s O{} m m}{%
  \IfBooleanT{#1}{\EE\nolimits\Bra*{#3}{#4}}%
  \IfBooleanF{#1}{\EE\nolimits\Bra[#2]{#3}{#4}}%
}

\usepackage{CJKutf8}

\newcommand{\mathInd}{\hphantom{{}={}}}
\newcommand{\by}[2][]{\text{\pbox[c]{\textwidth}{(by \pbox[t]{\textwidth}{\,\!#2)#1}}}}
\newcommand{\email}[1]{\href{mailto:#1}{\nolinkurl{#1}}}

\declaretheoremstyle[
  shaded={bgcolor=gray!15},
]{thmsty}

\declaretheoremstyle[
]{thmstynoshade}

\declaretheorem[
  name=Theorem,
  refname={Theorem,Theorems},
  style=thmsty,
]{theorem}

\declaretheorem[
  name=Proposition,
  refname={Proposition,Propositions},
  style=thmsty,
]{proposition}
\declaretheorem[
  name=Lemma,
  refname={Lemma,Lemmas},
  style=thmsty,
]{lemma}
\declaretheorem[
  name=Definition,
  refname={Definition,Definitions},
  style=thmsty,
]{definition}
\declaretheorem[
  name=Assumption,
  refname={Assumption,Assumptions},
  style=thmsty,
]{assumption}
\declaretheorem[
  name=Remark,
  refname={Remark,Remarks},
  style=thmstynoshade,
]{remark}

\crefname{algorithm}{Algorithm}{Algorithms}
\crefname{line}{Line}{Lines}
\crefname{section}{Section}{Sections}
\crefname{appendix}{Appendix}{Appendices}
\crefname{table}{Table}{Tables}
\crefname{figure}{Figure}{Figures}
\crefname{equation}{}{}
\Crefname{equation}{Eq.}{Eqs.}

\setlist[itemize]{
  topsep=0.4\baselineskip,
  itemsep=0\baselineskip,
  leftmargin=1.5em,
}

\setlist[enumerate]{
  font=\upshape,
  label=(\alph*),
  ref=(\alph*),
  topsep=0.4\baselineskip,
  itemsep=0\baselineskip,
  leftmargin=2em,
}

\newlist{enuminasm}{enumerate}{1} 
\setlist[enuminasm]{
  font=\upshape,
  label=(\alph*),
  ref=\theassumption(\alph*),
  topsep=0.4\baselineskip,
  itemsep=0\baselineskip,
  leftmargin=2em,
}
\crefalias{enuminasmi}{assumption}

\newlist{enuminthm}{enumerate}{1}
\setlist[enuminthm]{
  font=\upshape,
  label=(\alph*),
  ref=\thetheorem(\alph*),
  topsep=0.4\baselineskip,
  itemsep=0\baselineskip,
  leftmargin=2em,
}
\crefalias{enuminthmi}{theorem}

\newlist{enuminlem}{enumerate}{1}
\setlist[enuminlem]{
  font=\upshape,
  label=(\alph*),
  ref=\thelemma(\alph*),
  topsep=0.4\baselineskip,
  itemsep=0\baselineskip,
  leftmargin=2em,
}
\crefalias{enuminlemi}{lemma}

\newlist{enuminprop}{enumerate}{1}
\setlist[enuminprop]{
  font=\upshape,
  label=(\alph*),
  ref=\theproposition(\alph*),
  topsep=0.4\baselineskip,
  itemsep=0\baselineskip,
  leftmargin=2em,
}
\crefalias{enuminpropi}{proposition}

\newlist{enumincond}{enumerate}{1}
\setlist[enumincond]{
  font=\upshape,
  label=(\alph*),
  ref=\thecondition(\alph*),
  topsep=0.4\baselineskip,
  itemsep=0\baselineskip,
  leftmargin=2em,
}
\crefalias{enumincondi}{condition}

\newcommand{\xinit}{x_{\mathrm{init}}}
\newcommand{\Ltrue}{L}
\newcommand{\Htrue}[1]{H_{#1}}
\newcommand{\Lest}{\ell}
\newcommand{\Hest}{h}
\newcommand{\Lmax}{\bar{\Lest}}
\newcommand{\Linit}{\Lest_{\mathrm{init}}}
\newcommand{\xbest}{x^\star}

\newcommand{\Proposed}{\textsf{Proposed}\xspace}
\newcommand{\GD}{\textsf{GD}\xspace}
\newcommand{\JNJ}{\textsf{JNJ2018}\xspace}
\newcommand{\LL}{\textsf{LL2022}\xspace}
\newcommand{\MT}{\textsf{MT2022}\xspace}
\newcommand{\LBFGS}{\textsf{L-BFGS}\xspace}

\newcommand{\reviset}[1]{#1}
\newcommand{\revise}[1]{#1}

\usepackage{datetime}
\date{\vspace{-2.5\baselineskip}}

\author[1]{Naoki Marumo\footnote{Corresponding author. E-mail: \email{marumo@mist.i.u-tokyo.ac.jp}}}
\author[1,2]{Akiko Takeda}

\affil[1]{Graduate School of Information Science and Technology, University of Tokyo, Tokyo, Japan}
\affil[2]{Center for Advanced Intelligence Project, RIKEN, Tokyo, Japan}

\title{Universal heavy-ball method for nonconvex optimization under H\"older continuous Hessians}

\begin{document}
\maketitle

\begin{abstract}
  We propose a new first-order method for minimizing nonconvex functions with Lipschitz continuous gradients and H\"older continuous Hessians. The proposed algorithm is a heavy-ball method equipped with two particular restart mechanisms. It finds a solution where the gradient norm is less than $\epsilon$ in $O(H_{\nu}^{\frac{1}{2 + 2 \nu}} \epsilon^{- \frac{4 + 3 \nu}{2 + 2 \nu}})$ function and gradient evaluations, where $\nu \in [0, 1]$ and $H_{\nu}$ are the H\"older exponent and constant, respectively. This complexity result covers the classical bound of $O(\epsilon^{-2})$ for $\nu = 0$ and the state-of-the-art bound of $O(\epsilon^{-7/4})$ for $\nu = 1$. Our algorithm is $\nu$-independent and thus universal; it automatically achieves the above complexity bound with the optimal $\nu \in [0, 1]$ without knowledge of $H_{\nu}$. In addition, the algorithm does not require other problem-dependent parameters as input, including the gradient's Lipschitz constant or the target accuracy $\epsilon$. Numerical results illustrate that the proposed method is promising.
\end{abstract}

\smallskip
\begin{description}[leftmargin=!,labelwidth=\widthof{\bfseries Keywords:}]
  \item[Keywords:]
  First-order method,
  Heavy-ball method,
  Momentum method,
  Worst-case complexity,
  H\"older continuity
  \item[MSC2020:]
  90C26, 90C30, 65K05, 90C06
\end{description}
\smallskip

\section{Introduction}
\label{sec:introduction}
This paper studies general nonconvex optimization problems:
\begin{align}
  \min_{x \in \R^d} \ 
  f(x),
\end{align}
where $f \colon \R^d \to \R$ is twice differentiable and lower bounded, i.e., $\inf_{x \in \R^d} f(x) > - \infty$.
Throughout the paper, we impose the following assumption of Lipschitz continuous gradients.
\begin{assumption}
  \label{asm:gradient_lip}
  There exists a constant $L > 0$ such that $\norm*{\nabla f(x) - \nabla f(y)} \leq \Ltrue \norm*{x - y}$ for all $x, y \in \R^d$.
\end{assumption}

\emph{First-order methods} \citep{beck2017first,lan2020first}, which access $f$ through function and gradient evaluations, have gained increasing attention because they are suitable for large-scale problems.
A classical result is that the gradient descent method finds an $\epsilon$-stationary point (i.e., $x \in \R^d$ where $\norm*{\nabla f(x)} \leq \epsilon$) in $O(\epsilon^{-2})$ function and gradient evaluations under \cref{asm:gradient_lip}.
Recently, more sophisticated first-order methods have been developed to achieve faster convergence for more smooth functions.
Such methods \reviset{\citep{carmon2017convex,xu2017neon+,allen2018neon2,jin2018accelerated,li2022restarted,marumo2022parameter,li2023restarted}} have complexity bounds of $O(\epsilon^{-7/4})$ or $\tilde O(\epsilon^{-7/4})$ under Lipschitz continuity of Hessians in addition to gradients.%
\footnote{
  The $\tilde O$-notation hides polylogarithmic factors in $\epsilon^{-1}$.
  For example, the method of \citep{jin2018accelerated} has a complexity bound of $O(\epsilon^{-7/4} (\log \epsilon^{-1})^6)$.
}

This research stream raises two natural questions:
\begin{enumerate}
  [
    label={Question \arabic*.},
    ref={\arabic*},
    leftmargin=*
  ]
  \item
  \label{item:question_interpolation}
  How fast can first-order methods converge under smoothness assumptions stronger than Lipschitz continuous gradients but weaker than Lipschitz continuous Hessians?
  \item
  \label{item:question_universality}
  Can a \emph{single} algorithm achieve both of the following complexity bounds:
  $O(\epsilon^{-2})$ for functions with Lipschitz continuous gradients and
  $O(\epsilon^{-7/4})$ for functions with Lipschitz continuous gradients and Hessians?
\end{enumerate}
Question~\ref{item:question_universality} is also crucial from a practical standpoint because it is often challenging for users of optimization methods to check whether a function of interest has a Lipschitz continuous Hessian.
It would be nice if there were no need to use several different algorithms to achieve faster convergence.

Motivated by the questions, we propose a new first-order method and provide its complexity analysis with the \emph{H\"older continuity} of Hessians.
H\"older continuity generalizes Lipschitz continuity and has been widely used for complexity analyses of optimization methods \citep{devolder2014first,lan2015bundle,nesterov2015universal,grapiglia2017regularized,cartis2017worst,cartis2019universal,dvurechensky2017gradient,ghadimi2019generalized,grapiglia2019accelerated,grapiglia2020tensor}.
Several properties and an example of H\"older continuity can be found in \citep[Section~2]{grapiglia2017regularized}.
\begin{definition}
  \label{def:holder_constant}
  The H\"older constant of $\nabla^2 f$ with exponent $\nu \in [0, 1]$ is defined by
  \begin{align}
    \Htrue{\nu}
    \coloneqq
    \sup_{x,y \in \R^d,\,x \neq y}
    \frac{\norm{\nabla^2 f(x) - \nabla^2 f(y)}}{\norm*{x - y}^\nu}.
    \label{eq:def_holder_constant}
  \end{align}
  The Hessian $\nabla^2 f$ is said to be H\"older continuous with exponent $\nu$, or \emph{$\nu$-H\"older}, if $H_\nu < +\infty$.
\end{definition}
We should emphasize that $f$ determines the value of $\Htrue{\nu}$ for each $\nu \in [0, 1]$ and that $\nu$ is not a constant determined by $f$.
Under \cref{asm:gradient_lip}, we have $\Htrue{0} \leq 2 \Ltrue$ because the assumption implies $\norm{\nabla^2 f(x)} \leq \Ltrue$ for all $x \in \R^d$ \citep[Lemma~1.2.2]{nesterov2004introductory}.
For $\nu \in (0, 1]$, we may have $H_\nu = +\infty$, but we will allow it.
In contrast, all existing first-order methods \reviset{\citep{carmon2017convex,xu2017neon+,allen2018neon2,jin2018accelerated,li2022restarted,marumo2022parameter,li2023restarted}} with complexity bounds of $O(\epsilon^{-7/4})$ or $\tilde O(\epsilon^{-7/4})$ assume $\Htrue{1} < + \infty$ (i.e., the Lipschitz continuity of $\nabla^2 f$) in addition to \cref{asm:gradient_lip}.
We should note that it is often difficult to compute the H\"older constant $\Htrue{\nu}$ of a real-world function for a given $\nu \in [0, 1]$.

The proposed first-order method is a heavy-ball method equipped with two particular restart mechanisms, enjoying the following advantages:
\begin{itemize}
  \item 
  For $\nu \in [0, 1]$ \revise{such that $\Htrue{\nu} < + \infty$}, our algorithm finds an $\epsilon$-stationary point in
  \begin{align}
    O \prn[\Big]{
      \Htrue{\nu}^{\frac{1}{2 + 2 \nu}} \epsilon^{- \frac{4 + 3 \nu}{2 + 2 \nu}}
    }
    \label{eq:complexity_intro}
  \end{align}
  function and gradient evaluations under \cref{asm:gradient_lip}.
  This result answers Question~\ref{item:question_interpolation} and covers the classical bound of $O(\epsilon^{-2})$ for $\nu = 0$ and the state-of-the-art bound of $O(\epsilon^{-7/4})$ for $\nu = 1$.
  \item
  The complexity bound~\cref{eq:complexity_intro} is simultaneously attained for \emph{all} $\nu \in [0, 1]$ \revise{such that $\Htrue{\nu} < + \infty$} by a single $\nu$-independent algorithm.
  The algorithm thus automatically achieves the bound with the optimal $\nu \in [0, 1]$ that minimizes \cref{eq:complexity_intro}.
  This result affirmatively answers Question~\ref{item:question_universality}.
  \item
  Our algorithm requires no knowledge of problem-dependent parameters, including the optimal $\nu$, the Lipschitz constant $\Ltrue$, or the target accuracy $\epsilon$.
\end{itemize}

Let us describe our ideas for developing such an algorithm.
We employ the Hessian-free analysis recently developed for Lipschitz continuous Hessians~\citep{marumo2022parameter} to estimate the Hessian's H\"older continuity with only first-order information.
The Hessian-free analysis uses inequalities that include the Hessian's Lipschitz constant $\Htrue{1}$ but not a Hessian matrix itself, enabling us to estimate $\Htrue{1}$.
Extending this analysis to general $\nu$ allows us to estimate the H\"older constant $H_{\nu}$, given $\nu \in [0, 1]$.
We thus obtain an algorithm that requires $\nu$ as input and has the complexity bound~\cref{eq:complexity_intro} for the given $\nu$.
However, the resulting algorithm lacks usability because $\nu$ that minimizes \cref{eq:complexity_intro} is generally unknown.

Our main idea for developing a $\nu$-independent algorithm is to set $\nu = 0$ for the above $\nu$-dependent algorithm.
This may seem strange, but we prove that it works; a carefully designed algorithm for $\nu = 0$ achieves the complexity bound~\cref{eq:complexity_intro} for any $\nu \in [0, 1]$.
Although we design an estimate for $\Htrue{0}$, it also has a relationship with $\Htrue{\nu}$ for $\nu \in (0, 1]$, as will be stated in \cref{prop:Hest_upperbound}.
This proposition allows us to obtain the desired complexity bounds without specifying $\nu$.


To evaluate the numerical performance of the proposed method, we conducted experiments with standard machine-learning tasks.
The results illustrate that the proposed method outperforms state-of-the-art methods.

\paragraph{Notation.}
For vectors $a, b \in \R^d$, let $\inner{a}{b}$ denote the dot product and $\norm{a}$ denote the Euclidean norm.
For a matrix $A \in \R^{m \times n}$, let $\norm{A}$ denote the operator norm, or equivalently the largest singular value.

\section{Related work}
This section reviews previous studies from several perspectives and discusses similarities and differences between them and this work.

\paragraph{Complexity of first-order methods.}
Gradient descent (GD) is known to have a complexity bound of $O(\epsilon^{-2})$ under Lipschitz continuous gradients (e.g., \citep[Example.~1.2.3]{nesterov2004introductory}).
First-order methods \citep{cartis2017worst,ghadimi2019generalized} for H\"older continuous gradients have recently been proposed to generalize the bound; they enjoy bounds of $O(\epsilon^{-\frac{1+\mu}{\mu}})$, where $\mu \in (0, 1]$ is the H\"older exponent of $\nabla f$.
First-order methods have also been studied under stronger assumptions.
The methods of \citep{carmon2017convex,xu2017neon+,allen2018neon2,jin2018accelerated} enjoy complexity bounds of $\tilde O(\epsilon^{-7/4})$ under Lipschitz continuous gradients and Hessians,\footnote{
  \revise{
    We note that some methods \citep{carmon2018accelerated,agarwal2017finding,royer2018complexity,royer2020newton} also attain the complexity of $\tilde O(\epsilon^{-7/4})$, but they employ Hessian-vector multiplications and thus are beyond first-order methods.
  }
}
and the bounds have recently been improved to $O(\epsilon^{-7/4})$ \reviset{\citep{li2022restarted,marumo2022parameter,li2023restarted}}.
This paper generalizes the classical bound of $O(\epsilon^{-2})$ in a different direction from \citep{cartis2017worst,ghadimi2019generalized} and interpolates the existing bounds of $O(\epsilon^{-2})$ and $O(\epsilon^{-7/4})$.
\cref{table:complexity_exponents} compares our complexity results with the existing ones.

\begin{table}[t]
  \centering
  \caption{
    Complexity of first-order methods for nonconvex optimization.
    ``Exponent in complexity'' means $\alpha$ in $O(\epsilon^{- \alpha})$.
  }
  \label{table:complexity_exponents}
  \def\arraystretch{1.1}
  \begin{tabular}{@{}l|ccccc@{}}\toprule
    $\mu$-H\"older gradient & $\mu \in (0, 1]$       & $\mu = 1$   & $\mu = 1$                      & $\mu = 1$\\
    $\nu$-H\"older Hessian  & ---                    & ($\nu = 0$) & $\nu \in [0, 1]$               & $\nu = 1$\\\midrule
    Exponent in complexity  & $\dfrac{1 + \mu}{\mu}$ & $2$         & $\dfrac{4 + 3 \nu}{2 + 2 \nu}$ & $\dfrac{7}{4}$ \\
    Reference / Algorithm
    & \citep{cartis2017worst,ghadimi2019generalized}
    & Gradient descent
    & \textbf{This work}
    & \reviset{\citep{li2022restarted,marumo2022parameter,li2023restarted}} \\
    \bottomrule
  \end{tabular}
\end{table}

\paragraph{Complexity of second-order methods using H\"older continuous Hessians.}
The H\"older continuity of Hessians has been used to analyze second-order methods.
\citet{grapiglia2017regularized} proposed a regularized Newton method that finds an $\epsilon$-stationary point in $O(\epsilon^{-\frac{2 + \nu}{1 + \nu}})$ evaluations of $f$, $\nabla f$, and $\nabla^2 f$, where $\nu \in [0, 1]$ is the H\"older exponent of $\nabla^2 f$.
The complexity bound generalizes previous $O(\epsilon^{-3/2})$ bounds under Lipschitz continuous Hessians \citep{nesterov2006cubic,cartis2011adaptive,cartis2011adaptive2,curtis2017trust}.
We make the same assumption of H\"older continuous Hessians as in \citep{grapiglia2017regularized} but do not compute Hessians in the algorithm.
\cref{table:first_second_order_methods_assumption} summarizes the first-order and second-order methods together with their assumptions.

\begin{table}[t]
  \bigskip
  \centering
  \caption{Nonconvex optimization methods under smoothness assumptions.}
  \label{table:first_second_order_methods_assumption}
  \def\arraystretch{1.1}
  \begin{tabular}{@{}ll|cc@{}}
    \toprule
    Assumption         &                    & First-order method  & Second-order method\\\midrule
    Lipschitz gradient & ($\mu = 1$)        & Gradient descent                                                                                           & --- \\
    H\"older gradient  & ($\mu \in (0, 1]$) & \citep{cartis2017worst,ghadimi2019generalized}                                                             & --- \\
    Lipschitz Hessian  & ($\nu = 1$)        & \reviset{\citep{carmon2017convex,xu2017neon+,allen2018neon2,jin2018accelerated,li2022restarted,marumo2022parameter,li2023restarted}} & \citep{nesterov2006cubic,cartis2011adaptive,cartis2011adaptive2,curtis2017trust} \\
    H\"older Hessian   & ($\nu \in [0, 1]$) & \textbf{This work}                                                                                         & \citep{grapiglia2017regularized} \\
    \bottomrule
  \end{tabular}
\end{table}

\paragraph{Universality for H\"older continuity.}
When H\"older continuity is assumed, it is preferable that algorithms not require the exponent $\nu$ as input because a suitable value for $\nu$ tends to be hard to find in real-world problems.
Such $\nu$-independent algorithms, called \emph{universal methods}, were first developed as first-order methods for convex optimization~\citep{nesterov2015universal,lan2015bundle} and have since been extended to other settings, including higher-order methods or nonconvex problems \citep{grapiglia2017regularized,cartis2017worst,cartis2019universal,dvurechensky2017gradient,ghadimi2019generalized,grapiglia2019accelerated,grapiglia2020tensor}.
Within this research stream, this paper proposes a universal method with a new setting: a first-order method under H\"older continuous Hessians.
Because of the differences in settings, the existing techniques for universality cannot be applied directly; we obtain a universal method by setting $\nu = 0$ for a $\nu$-dependent algorithm, as discussed in \cref{sec:introduction}.

\paragraph{Heavy-ball methods.}
Heavy-ball (HB) methods are a kind of momentum method first proposed by \citet{polyak1964methods} for convex optimization.
Although some complexity results have been obtained for (strongly) convex settings \citep{ghadimi2015global,lessard2016analysis}, they are weaker than the optimal bounds given by Nesterov's accelerated gradient method \citep{nesterov1983method,nesterov2018lectures}.
For nonconvex optimization, HB and its variants~\citep{sutskever2013importance,kingma2015adam,reddi2018convergence,cutkosky2020momentum} have been practically used with great success, especially in deep learning,
\revise{
while studies on theoretical convergence analysis are few \citep{oneill2019behavior,ochs2014ipiano,li2023restarted}.
\citet{oneill2019behavior} analyzed the local behavior of the original HB method, showing that the method is unlikely to converge to strict saddle points.
\citet{ochs2014ipiano} proposed a generalized HB method, iPiano, that enjoys a complexity bound of $O(\epsilon^{-2})$ under Lipschitz continuous gradients, which is of the same order as that of GD.
\citet{li2023restarted} proposed an HB method with a restart mechanism that achieves a complexity bound of $O(\epsilon^{-7/4})$ under Lipschitz continuous gradients and Hessians.
Our algorithm is another HB method with a different restart mechanism that enjoys more general complexity bounds than \citet{li2023restarted}, as discussed in \cref{sec:introduction}.
}

\paragraph{Comparison with \citep{marumo2022parameter}.}
This paper shares some mathematical tools with \citep{marumo2022parameter} because we utilize the Hessian-free analysis introduced in \citep{marumo2022parameter} to estimate Hessian's H\"older continuity.
\revise{While} the analysis in \citep{marumo2022parameter} is for Nesterov's accelerated gradient method under Lipschitz continuous Hessians, \revise{we here analyze Polyak's HB method under H\"older continuity.}
\revise{Thanks to the simplicity of the HB momentum,} our estimate for the H\"older constant \revise{is} easier to compute than the estimate for the Lipschitz constant proposed in~\citep{marumo2022parameter}, which improves the efficiency of our algorithm.
\revise{
We would like to emphasize that a $\nu$-independent algorithm cannot be derived simply by applying the mathematical tools in \citep{marumo2022parameter}.
It should also be mentioned that we have not confirmed that it is impossible or very challenging to develop a $\nu$-independent algorithm with Nesterov's momentum under H\"older continuous Hessians.
}

\revise{
\paragraph{Lower bounds.}
So far, we have discussed upper bounds on complexity, but there are also some studies on its lower bounds.
\citet{carmon2020lower} proved that no deterministic or stochastic first-order method can improve the complexity of $O(\epsilon^{-2})$ with the assumption of Lipschitz continuous gradients alone.
(See \citep[Theorems~1 and 2]{carmon2020lower} for more rigorous statements.)
This result implies that GD is optimal in terms of complexity under Lipschitz continuous gradients.
\citet{carmon2021lower} showed a lower bound of $\Omega(\epsilon^{-12/7})$ for first-order methods under Lipschitz continuous gradients and Hessians.
Compared with the upper bound of $O(\epsilon^{-7/4})$ under the same assumptions, there is still a $\Theta(\epsilon^{-1/28})$ gap.
Closing this gap would be an interesting research question, though this paper does not focus on it.
}

\section{Preliminary results}
The following lemma is standard for the analyses of first-order methods.
\begin{lemma}[{e.g., \citep[Lemma 1.2.3]{nesterov2004introductory}}]
  \label{lem:eq:property_gradlip_obj}
  Under \cref{asm:gradient_lip}, the following holds for any $x, y \in \R^d$:
  \begin{align}
    f(x) - f(y)
    &\leq
    \inner*{\nabla f(y)}{x - y}
    + \frac{\Ltrue}{2} \norm*{x - y}^2.
  \end{align}
\end{lemma}
This inequality helps estimate the Lipschitz constant $\Ltrue$ and evaluate the decrease in the objective function per iteration.

We also use the following inequalities derived from H\"older continuous Hessians.
\begin{lemma}
  \label{lem:gradient_jensen}
  For any $z_1, \dots, z_n \in \R^d$, let $\bar z \coloneqq \sum_{i=1}^n \lambda_i z_i$, where $\lambda_1,\dots,\lambda_n \geq 0$ and $\sum_{i=1}^n \lambda_i = 1$.
  Then, the following holds for all $\nu \in [0, 1]$ \revise{such that $\Htrue{\nu} < + \infty$}:
  \begin{align}
    \norm*{
      \nabla f(\bar z)
      - \sum_{i=1}^n \lambda_i \nabla f \prn*{ z_i }
    }
    \leq
    \frac{\Htrue{\nu}}{1 + \nu} \sum_{i=1}^n \lambda_i \norm*{z_i - \bar z}^{1 + \nu}
    \leq
    \frac{\Htrue{\nu}}{1 + \nu}
    \prn[\Bigg]{
      \sum_{1 \leq i < j \leq n} \lambda_i \lambda_j \norm*{z_i - z_j}^2
    }^{\frac{1 + \nu}{2}}.
  \end{align}
\end{lemma}
\begin{lemma}
  \label{lem:trapezoidal_rule_error}
  For all \revise{$x, y \in \R^d$ and $\nu \in [0, 1]$ such that $\Htrue{\nu} < + \infty$}, the following holds:
  \begin{align}
    f(x) - f(y)
    &\leq
    \frac{1}{2} \inner*{\nabla f(x) + \nabla f(y)}{x - y}
    + \frac{2 \Htrue{\nu}}{(1 + \nu) (2 + \nu) (3 + \nu)} \norm*{x - y}^{2 + \nu}.
  \end{align}
\end{lemma}
The proofs are given in \cref{sec:proof_lem_holder_hessian}.
These lemmas generalize \citep[Lemmas 2 and 3]{marumo2022parameter} for Lipschitz continuous Hessians (i.e., $\nu = 1$).
It is important to note that the inequalities in \cref{lem:gradient_jensen,lem:trapezoidal_rule_error} are Hessian-free; they include the Hessian's H\"older constant $\Htrue{\nu}$ but not a Hessian matrix itself.
Accordingly, we can adaptively estimate the H\"older continuity of $\nabla^2 f$ in the algorithm without computing Hessians.

\section{Algorithm}
The proposed method, \cref{alg:proposed}, is a heavy-ball (HB) method equipped with two particular restart schemes.
In the algorithm, the iteration counter $k$ is reset to $0$ when HB restarts on \cref{alg-line-hb:restart_unsuccessful} or \ref{alg-line-hb:restart_successful}, whereas the total iteration counter $K$ is not.
We refer to the period between one reset of $k$ and the next reset as an \emph{epoch}.
Note that it is unnecessary to implement $K$ in the algorithm; it is included here only to make the statements in our analysis concise.

The algorithm uses estimates $\Lest$ and $\Hest_k$ for the Lipschitz constant $\Ltrue$ and the H\"older constant $\Htrue{0}$.
The estimate $\Lest$ is fixed during an epoch, while $\Hest_k$ is updated at each iteration, having the subscript $k$.

\subsection{Update of solutions}
With an estimate $\Lest$ for the Lipschitz constant $\Ltrue$, \cref{alg:proposed} defines a solution sequence $(x_k)$ as follows: $v_0 = \0$ and
\begin{align}
  v_k &= \theta_{k-1} v_{k-1} - \frac{1}{\Lest} \nabla f(x_{k-1}),
  \label{eq:update_v}\\
  x_k &= x_{k-1} + v_k
  \label{eq:update_x}
\end{align}
for $k \geq 1$.
Here, $(v_k)$ is the velocity sequence, and $0 \leq \theta_k \leq 1$ is the momentum parameter.
Let $x_{-1} \coloneqq x_0$ for convenience, which makes \cref{eq:update_x} valid for $k = 0$.
This type of optimization method is called a heavy-ball method or Polyak's momentum method.

In this paper, we use the simplest parameter setting:
\begin{align}
  \theta_k = 1
\end{align}
for all $k \geq 1$.
\revise{
  Our choice of $\theta_k$ differs from the existing ones; the existing complexity analyses \citep{polyak1964methods,lessard2016analysis,danilova2020non,ghadimi2015global,danilova2021averaged,li2023restarted} of HB prohibit $\theta_k = 1$.
  For example, \citet{li2023restarted} proposed $\theta_k = 1 - 5 (\Htrue{1} \epsilon)^{1/4} / \sqrt{\Ltrue}$.
}
Our new proof technique described later in \cref{sec:analysis_objective_decrease} enables us to set $\theta_k = 1$.

We will later use the averaged solution
\begin{align}
  \bar x_k
  \coloneqq
  \frac{1}{k}
  \sum_{i=0}^{k-1} x_i
  \label{eq:def_xbar}
\end{align}
to compute the estimate $h_k$ for $\Htrue{0}$ and set the best solution $\xbest_k$.
The averaged solution can be computed efficiently with a simple recursion: $\bar x_{k+1} = \frac{k}{k+1} \bar x_k + \frac{1}{k+1} x_k$.

\begin{algorithm}[t]
  \caption{Proposed heavy-ball method\label{alg:proposed}}
  \begin{algorithmic}[1]
  \Require{%
    $\xinit \in \R^d$,
    $\Linit > 0$,
    $\alpha > 1$,
    $0 < \beta \leq 1$.
    Recommended: $(\Linit, \alpha, \beta) = (10^{-3}, 2, 0.1)$
  }
  \State{%
    $(x_0, v_0) \gets (\xinit, \0)$,
    $\Lest \gets \Linit$,
    $k \gets 0$,
    $K \gets 0$
  }
  \Repeat
    \State{%
      $k \gets k + 1$,
      $K \gets K + 1$
    }
    \label{alg-line-hb:update_k}
    \State $v_k \gets v_{k-1} - \frac{1}{\Lest} \nabla f(x_{k-1})$
    \label{alg-line-hb:update_v}
    \State $x_k \gets x_{k-1} + v_k$
    \label{alg-line-hb:update_x}
    \State $\xbest_k \gets \argmin_{x \in \set{x_0,\dots,x_k,\bar x_1,\dots,\bar x_k}} f(x)$
    \Comment{$\bar x_k$ is defined by \cref{eq:def_xbar}}
    \If{Condition \cref{eq:armijo_rule} does not hold}
    \label{alg-line-hb:check_decrease}
      \State{%
        $(x_0, v_0) \gets (\xbest_k, \0)$,
        $\Lest \gets \alpha \Lest$,
        $k \gets 0$
      }
      \label{alg-line-hb:restart_unsuccessful}
    \ElsIf{$k (k+1) \Hest_k > \frac{3}{8} \Lest$}
    \label{alg-line-hb:check_large_k}
    \Comment{$\Hest_k$ is defined by \cref{eq:def_Hest}}
      \State{%
        $(x_0, v_0) \gets (\xbest_k, \0)$,
        $\Lest \gets \beta \Lest$,
        $k \gets 0$
      }
      \label{alg-line-hb:restart_successful}
    \EndIf
    \Until{convergence}
    \State \Return $\xbest_k$
  \end{algorithmic}
\end{algorithm}

\subsection{Estimation of H\"older continuity}
\label{sec:estimation_holder}
Let 
\begin{align}
  S_k
  \coloneqq
  \sum_{i=1}^k \norm*{v_i}^2
  \label{eq:def_Sk}
\end{align}
to simplify the notation.
Our analysis uses the following inequalities due to \cref{lem:trapezoidal_rule_error,lem:gradient_jensen}.
\begin{lemma}
  \label{lem:two_inequalities_for_analysis_holder}
  For all \revise{$k \geq 1$ and $\nu \in [0, 1]$ such that $\Htrue{\nu} < + \infty$}, the following hold:
  \begin{align}
    f(x_k) - f(x_{k-1})
    &\leq
    \frac{1}{2} \inner{\nabla f(x_{k-1}) + \nabla f(x_k)}{v_k}
    + \frac{2 \Htrue{\nu}}{(1 + \nu) (2 + \nu) (3 + \nu)} \norm*{v_k}^{2 + \nu},
    \label{eq:f_decrease_bound_holder}\\
    \norm*{ \nabla f (\bar x_k) }
    &\leq
    \frac{\Lest}{k} \norm*{v_k}
    + \frac{\Htrue{\nu}}{1 + \nu} \prn*{ \frac{k S_k}{8} }^{\frac{1 + \nu}{2}}.
    \label{eq:grad_norm_xbar_upperbound}
  \end{align}
\end{lemma}
\begin{proof}
  \Cref{eq:f_decrease_bound_holder} follows immediately from \cref{lem:trapezoidal_rule_error}.
  The proof of \cref{eq:grad_norm_xbar_upperbound} is given in \cref{sec:proof_eq_grad_norm_xbar_upperbound}.
\end{proof}

\cref{alg:proposed} requires no information on the H\"older continuity of $\nabla^2 f$, automatically estimating it.
To illustrate the trick, let us first consider a prototype algorithm that works when a value of $\nu \in [0, 1]$ such that $\Htrue{\nu} < + \infty$ is given.
Given such a $\nu$, one can compute an estimate $h$ of $\Htrue{\nu}$ such that
\begin{align}
  f(x_k) - f(x_{k-1})
  &\leq
  \frac{1}{2} \inner{\nabla f(x_{k-1}) + \nabla f(x_k)}{v_k}
  + \frac{2 \Hest}{(1 + \nu) (2 + \nu) (3 + \nu)} \norm*{v_k}^{2 + \nu},
  \label{eq:f_decrease_bound_holder_nu}\\
  \norm*{ \nabla f (\bar x_k) }
  &\leq
  \frac{\Lest}{k} \norm*{v_k}
  + \frac{\Hest}{1 + \nu} \prn*{ \frac{k S_k}{8} }^{\frac{1 + \nu}{2}},
  \label{eq:grad_norm_xbark_upperbound_nu}
\end{align}
which come from \cref{lem:two_inequalities_for_analysis_holder}.
This estimation scheme yields a $\nu$-dependent algorithm that has the complexity bound~\cref{eq:complexity_intro} for the given $\nu$, though we will omit the details.
The algorithm is not so practical because it requires $\nu \in [0, 1]$ such that $\Htrue{\nu} < + \infty$ as input.
However, perhaps surprisingly, setting $\nu = 0$ for the $\nu$-dependent algorithm gives a $\nu$-independent algorithm that achieves the bound~\cref{eq:complexity_intro} for \emph{all} $\nu \in [0, 1]$.
\cref{alg:proposed} is the $\nu$-independent algorithm obtained in that way.

Let $\Hest_0 \coloneqq 0$ for convenience.
At iteration $k \geq 1$ of each epoch, we use the estimate $\Hest_k$ for $\Htrue{0}$ defined by
\begin{align}
  \Hest_k
  =
  \max \bigg\{
    h_{k-1},\ 
    &
    \frac{3}{\norm*{v_k}^2}
    \prn[\Big]{
      f(x_k) - f(x_{k-1})
      - \frac{1}{2} \inner{\nabla f(x_{k-1}) + \nabla f(x_k)}{v_k}
    }
    ,\\
    &
    \sqrt{\frac{8}{k S_k}}
    \prn*{ \norm*{ \nabla f (\bar x_k) } - \frac{\Lest}{k} \norm*{v_k} }
  \bigg\}
  \label{eq:def_Hest}
\end{align}
so that $\Hest_k \geq \Hest_{k-1}$ and
\begin{align}
  f(x_k) - f(x_{k-1})
  &\leq
  \frac{1}{2} \inner{\nabla f(x_{k-1}) + \nabla f(x_k)}{v_k}
  + \frac{h_k}{3} \norm*{v_k}^2,
  \label{eq:f_decrease_bound_holder_Hestk}\\
  \norm*{ \nabla f (\bar x_k) }
  &\leq
  \frac{\Lest}{k} \norm*{v_k}
  + \Hest_k \sqrt{\frac{k S_k}{8}}.
  \label{eq:grad_norm_xbark_upperbound_Hestk}
\end{align}
The above inequalities were obtained by plugging $\nu = 0$ into \cref{eq:f_decrease_bound_holder_nu,eq:grad_norm_xbark_upperbound_nu}.

Although we designed $\Hest_k$ to estimate $\Htrue{0}$, it fortunately also relates to $\Htrue{\nu}$ for general $\nu \in [0, 1]$.
The following upper bound on $\Hest_k$ shows the relationship between $\Hest_k$ and $\Htrue{\nu}$, which will be used in the complexity analysis.
\begin{proposition}
  \label{prop:Hest_upperbound}
  For all \revise{$k \geq 1$ and $\nu \in [0, 1]$ such that $\Htrue{\nu} < + \infty$}, the following holds:
  \begin{align}
    \Hest_k
    \leq
    \Htrue{\nu} (k S_k)^{\frac{\nu}{2}}.
  \end{align}
\end{proposition}
\begin{proof}
  \cref{lem:two_inequalities_for_analysis_holder} gives
  \begin{align}
    \frac{3}{\norm*{v_k}^2}
    \prn[\Big]{
      f(x_k) - f(x_{k-1})
      - \frac{1}{2} \inner{\nabla f(x_{k-1}) + \nabla f(x_k)}{v_k}
    }
    &\leq
    \frac{6 \Htrue{\nu}}{(1 + \nu) (2 + \nu) (3 + \nu)}
    \norm*{v_k}^{\nu},\\
    \sqrt{\frac{8}{k S_k}}
    \prn*{ \norm*{ \nabla f (\bar x_k) } - \frac{\Lest}{k} \norm*{v_k} }
    &\leq
    \frac{\Htrue{\nu}}{1 + \nu} \prn*{ \frac{k S_k}{8} }^{\frac{\nu}{2}}
    \leq
    \frac{\Htrue{\nu}}{1 + \nu} (k S_k)^{\frac{\nu}{2}}.
  \end{align}
  Hence, definition \cref{eq:def_Hest} of $\Hest_k$ yields
  \begin{alignat}{2}
    \Hest_k
    &\leq
    \max \set*{
      h_{k-1},\,
      \frac{6 \Htrue{\nu}}{(1 + \nu) (2 + \nu) (3 + \nu)}
      \norm*{v_k}^{\nu}
      ,\,
      \frac{\Htrue{\nu}}{1 + \nu}
      (k S_k)^{\frac{\nu}{2}}
    }\\
    &\leq
    \max \set*{
      h_{k-1},\,
      \Htrue{\nu} \norm*{v_k}^{\nu}
      ,\,
      \Htrue{\nu} (k S_k)^{\frac{\nu}{2}}
    }
    &\quad&\by{$\nu \geq 0$}\\
    &=
    \max \set*{
      h_{k-1},\,
      \Htrue{\nu} \prn*{k S_k}^{\frac{\nu}{2}}
    }
    &\quad&\by{$\norm*{v_k} \leq \sqrt{S_k} \leq \sqrt{k S_k}$}.
  \end{alignat}
  The desired result follows inductively since $\Htrue{\nu} (k S_k)^{\frac{\nu}{2}}$ is nondecreasing in $k$.
\end{proof}

For $\nu = 0$, \cref{prop:Hest_upperbound} gives a natural upper bound, $\Hest_k \leq \Htrue{0}$, since the estimate $\Hest_k$ is designed for $\Htrue{0}$ based on \cref{lem:two_inequalities_for_analysis_holder}.
For $\nu \in (0, 1]$, the upper bound can become tighter when $k S_k$ is small.
Indeed, the iterates $(x_k)$ are expected to move less significantly in an epoch as the algorithm proceeds.
\revise{Accordingly, $(S_k)$ increases more slowly in later epochs,} yielding a tighter upper bound on $\Hest_k$.
This trick improves the complexity bound from $O(\epsilon^{-2})$ for $\nu = 0$ to $O(\epsilon^{- \frac{4 + 3 \nu}{2 + 2 \nu}})$ for general $\nu \in [0, 1]$.

\subsection{Restart mechanisms}
\cref{alg:proposed} is equipped with two restart mechanisms.
The first one uses the standard descent condition
\begin{align}
  f(x_k) - f(x_{k-1})
  \leq
  \inner{\nabla f(x_{k-1})}{v_k} + \frac{\Lest}{2} \norm*{v_k}^2
  \label{eq:armijo_rule}
\end{align}
to check whether the current estimate $\Lest$ for the Lipschitz constant $\Ltrue$ is large enough.
If the descent condition \cref{eq:armijo_rule} does not hold, HB restarts with a larger $\Lest$ from the best solution $\xbest_k \coloneqq \argmin_{x \in \set{x_0,\dots,x_k,\bar x_1,\dots,\bar x_k}} f(x)$ during the epoch.
\revise{
  We consider not only $x_0,\dots,x_k$ but also the averaged solutions $\bar x_1,\dots,\bar x_k$ as candidates for the next starting point because averaging may stabilize the behavior of the HB method.
  As we will show later in \cref{lem:min_grad_xbar_upperbound2}, the gradient norm of averaged solutions is small, which leads to stability.
  For strongly-convex quadratic problems, \citet{danilova2021averaged} also show that averaged HB methods have a smaller maximal deviation from the optimal solution than the vanilla HB method.
  A similar effect for nonconvex problems is expected in the neighborhood of local optima where quadratic approximation is justified.
}

The second restart scheme resets the momentum effect when $k$ becomes large; if
\begin{align}
  k (k+1) \Hest_k > \frac{3}{8} \Lest
  \label{eq:condition_restart}
\end{align}
is satisfied, HB restarts from the best solution $\xbest_k$.
At the restart, we can reset $\Lest$ to a smaller value in the hope of improving practical performance, though decreasing $\Lest$ is not necessary for the complexity analysis.
This restart scheme guarantees that 
\begin{align}
  k (k-1) \Hest_{k-1} \leq \frac{3}{8} \Lest
  \label{eq:continue_condition_k-1}
\end{align}
holds at iteration $k$ of each epoch.

The Lipschitz estimate $\Lest$ increases only when the descent condition~\cref{eq:armijo_rule} is violated.
On the other hand, \cref{lem:eq:property_gradlip_obj} implies that condition \cref{eq:armijo_rule} always holds as long as $\Lest \geq \Ltrue$.
Hence, we have the following upper bound on $\Lest$.
\begin{proposition}
  \label{prop:Lest_upperbound}
  Suppose that \cref{asm:gradient_lip} holds.
  Then, the following is true throughout \cref{alg:proposed}: 
  $\Lest \leq \max \set{\Linit, \alpha \Ltrue}$.
\end{proposition}

\section{Complexity analysis}
This section proves that \cref{alg:proposed} enjoys the complexity bound~\cref{eq:complexity_intro} for all $\nu \in [0, 1]$.

\subsection{Objective decrease for one epoch}
\label{sec:analysis_objective_decrease}
First, we evaluate the decrease in the objective function value during one epoch.
\begin{lemma}
  \label{lem:decrease_epoch}
  Suppose that \cref{asm:gradient_lip} holds and that the descent condition
  \begin{align}
    f(x_i) - f(x_{i-1})
    \leq
    \inner{\nabla f(x_{i-1})}{v_i} + \frac{\Lest}{2} \norm*{v_i}^2
    \label{eq:armijo_rule_i}
  \end{align}
  holds for all $1 \leq i \leq k$.
  Then, the following holds under condition~\cref{eq:continue_condition_k-1}:
  \begin{align}
    \min_{1 \leq i \leq k} f(x_i)
    \leq
    f(x_0) - \frac{\Lest S_k}{4k}.
    \label{eq:decrease_epoch}
  \end{align}
\end{lemma}
Before providing the proof, let us remark on the lemma.

Evaluating the decrease in the objective function is the central part of a complexity analysis.
It is also an intricate part because the function value does not necessarily decrease monotonically in nonconvex acceleration methods.
To overcome the non-monotonicity, previous analyses have employed different proof techniques.
For example, \citet{li2022restarted} constructed a quadratic approximation of the objective, diagonalized the Hessian, and evaluated the objective decrease separately for each coordinate; \citet{marumo2022parameter} designed a tricky potential function and showed that it is nearly decreasing.

This paper uses another technique to deal with the non-monotonicity.
We observe that the solution $x_k$ does not need to attain a small function value; it is sufficient for at least one of $x_1,\dots,x_k$ to do so, thanks to our particular restart mechanism.
This observation permits the left-hand side of \cref{eq:decrease_epoch} to be $\min_{1 \leq i \leq k} f(x_i)$ rather than $f(x_k)$ and makes the proof easier.
The proof of \cref{lem:decrease_epoch} calculates a weighted sum of $2k-1$ inequalities derived from \cref{lem:eq:property_gradlip_obj,lem:trapezoidal_rule_error}, which is elementary compared with the existing proofs.
Now, we provide that proof.

\begin{proof}[Proof of \cref{lem:decrease_epoch}]
  Combining \cref{eq:armijo_rule_i} with the update rules \cref{eq:update_v,eq:update_x} yields
  \begin{alignat}{2}
    f(x_i) - f(x_{i-1})
    &\leq
    \inner{\nabla f(x_{i-1})}{v_i} + \frac{\Lest}{2} \norm*{v_i}^2
    =
    \Lest \inner{v_{i-1}}{v_i} - \frac{\Lest}{2} \norm*{v_i}^2
    \label{eq:f_decrease_easy}
  \end{alignat}
  for $1 \leq i \leq k$.
  For $1 \leq i < k$, we also have
  \begin{alignat}{2}
    f(x_i) - f(x_{i-1})
    &\leq
    \frac{1}{2} \inner{\nabla f(x_{i-1}) + \nabla f(x_i)}{v_i}
    + \frac{\Hest_{k-1}}{3} \norm*{v_i}^2
    &\quad&\by{\cref{eq:f_decrease_bound_holder_Hestk} and $h_i \leq h_{k-1}$}\\
    &=
    \frac{\Lest}{2} \inner{v_{i-1}}{v_i}
    - \frac{\Lest}{2} \inner{v_i}{v_{i+1}}
    + \frac{\Hest_{k-1}}{3} \norm*{v_i}^2
    &\quad&\by{\cref{eq:update_v}}.
    \label{eq:f_decrease_bound_holder_Hestk_k-1}
  \end{alignat}
  We will calculate a weighted sum of $2k-1$ inequalities:
  \begin{itemize}
    \item
    \cref{eq:f_decrease_easy} with weight $1$ for $1 \leq i \leq k$,
    \item
    \cref{eq:f_decrease_bound_holder_Hestk_k-1} with weight $2(k-i)$ for $1 \leq i < k$.
  \end{itemize}
  The left-hand side of the weighted sum is
  \begin{align}
    &\mathInd
    \sum_{i=1}^k \prn*{ f(x_i) - f(x_{i-1}) }
    + \sum_{i=1}^{k-1} 2(k-i) \prn*{ f(x_i) - f(x_{i-1}) }\\
    &=
    - (2k-1) f(x_0)
    + \sum_{i=1}^{k-1} 2 f(x_i)
    + f(x_k)
    \geq
    (2k-1) \prn*{
      \min_{1 \leq i \leq k} f(x_i)
      - f(x_0)
    }.
  \end{align}
  On the right-hand side of the weighted sum, some calculations with $v_0 = \0$ show that the inner-product terms of $\inner{v_{i-1}}{v_i}$ cancel out as follows:
  \begin{alignat}{2}
    &\mathInd
    \Lest \sum_{i=1}^k \inner{v_{i-1}}{v_i}
    + \Lest \sum_{i=1}^{k-1} (k-i) \prn*{ \inner{v_{i-1}}{v_i} - \inner{v_i}{v_{i+1}} }\\
    &=
    \Lest \sum_{\reviset{i=2}}^k \inner{v_{i-1}}{v_i}
    + \Lest \sum_{\reviset{i=2}}^{k-1} (k-i) \inner{v_{i-1}}{v_i}
    - \Lest \sum_{i=2}^k (k-i+1) \inner{v_{i-1}}{v_i}
    =
    0.
  \end{alignat}
  The remaining terms on the right-hand side of the weighted sum are
  \begin{alignat}{2}
    &\mathInd
    {- \frac{\Lest}{2}} \sum_{i=1}^k \norm*{v_i}^2
    + \frac{\Hest_{k-1}}{3}
    \sum_{i=1}^{k-1} 2 (k-i) \norm*{v_i}^2\\
    &\leq
    - \frac{\Lest}{2} \sum_{i=1}^k \norm*{v_i}^2
    + \frac{\Hest_{k-1}}{3}
    \sum_{i=1}^k 2 (k-1) \norm*{v_i}^2
    =
    - \prn*{ \frac{\Lest}{2} - \frac{2}{3} (k-1) \Hest_{k-1} } S_k.
  \end{alignat}
  We now obtain
  \begin{align}
    \min_{1 \leq i \leq k} f(x_i) - f(x_0)
    \leq 
    - \prn*{ \frac{\Lest}{2} - \frac{2}{3} (k-1) \Hest_{k-1} }
    \frac{S_k}{2k-1}.
  \end{align}
  Finally, we evaluate the coefficient on the right-hand side with \cref{eq:continue_condition_k-1} as
  \begin{align}
    \frac{\Lest}{2} - \frac{2}{3} (k-1) \Hest_{k-1}
    \geq
    \frac{\Lest}{2} - \frac{\Lest}{4 k}
    =
    \Lest \frac{2k-1}{4k},
    \label{eq:proof_decrease_epoch_final_step}
  \end{align}
  which completes the proof.
\end{proof}

The proof elucidates that the second restart condition~\cref{eq:condition_restart} was designed to derive the lower bound of $\Lest \frac{2k-1}{4k}$ \revise{in \cref{eq:proof_decrease_epoch_final_step}}.

For an epoch that ends at \cref{alg-line-hb:restart_successful} in iteration $k \geq 1$, \cref{lem:decrease_epoch} gives
\begin{align}
  f(\xbest_k)
  \leq
  \min_{1 \leq i \leq k} f(x_i)
  \leq
  f(x_0) - \frac{\Lest S_k}{4k}.
  \label{eq:decrease_epoch_successful}
\end{align}
For an epoch that ends at \cref{alg-line-hb:restart_unsuccessful} in iteration $k \geq 2$, the lemma gives
\begin{align}
  f(\xbest_k)
  \leq
  f(\xbest_{k-1})
  \leq
  \min_{1 \leq i \leq k-1} f(x_i)
  \leq
  f(x_0) - \frac{\Lest S_{k-1}}{4 (k-1)}
  \leq
  f(x_0) - \frac{\Lest S_{k-1}}{4k}.
  \label{eq:decrease_epoch_unsuccessful}
\end{align}
These bounds will be used to derive the complexity bound.

\subsection{Upper bound on gradient norm}
Next, we prove the following upper bound on the gradient norm at the averaged solution.
\begin{lemma}
  \label{lem:min_grad_xbar_upperbound2}
  In \cref{alg:proposed}, the following holds at iteration $k \geq 2$:
  \begin{align}
    \min_{1 \leq i < k}
    \norm*{ \nabla f (\bar x_i) }
    \leq
    \Lest \sqrt{\frac{8 S_{k-1}}{k^3}}.
  \end{align}
\end{lemma}
\begin{proof}
  For $k = 2$, the result follows from $\norm*{ \nabla f (\bar x_1) } = \norm*{ \nabla f (x_0) } = \Lest \norm*{v_1}$.
  Below, we assume that $k \geq 3$.
  Let $A_k \coloneqq \sum_{i=1}^{k-1} i^2$; we have
  \begin{align}
    A_k
    =
    \frac{k (k-1) (2k-1)}{6}
    \geq
    \frac{k^3}{6}
    \label{eq:Ak_lowerbound}
  \end{align}
  for $k \geq 3$.
  A weighted sum of \cref{eq:grad_norm_xbark_upperbound_Hestk} over $k$ yields
  \begin{align}
    A_k
    \min_{1 \leq i < k}
    \norm*{ \nabla f (\bar x_i) }
    &\leq
    \sum_{i=1}^{k-1}
    i^2 \norm*{ \nabla f (\bar x_i) }
    \leq
    \Lest
    \sum_{i=1}^{k-1} i \norm*{v_i}
    +
    \Hest_{k-1} \sqrt{S_{k-1}}
    \sum_{i=1}^{k-1} i^2 \sqrt{\frac{i}{8}}
  \end{align}
  since $h_k$ and $S_k$ are nondecreasing in $k$.
  Each term can be bounded by the Cauchy--Schwarz inequality as
  \begin{align}
    \sum_{i=1}^{k-1} i \norm*{v_i}
    &\leq
    \sqrt{A_k S_{k-1}},\quad
    \sum_{i=1}^{k-1} i^2 \sqrt{\frac{i}{8}}
    =
    \sum_{i=1}^{k-1} i \sqrt{\frac{i^3}{8}}
    \leq
    \sqrt{A_k}
    \prn*{ \sum_{i=1}^{k-1} \frac{i^3}{8} }^{1/2}
    =
    \sqrt{\frac{A_k}{32}} k (k-1),
  \end{align}
  and thus
  \begin{alignat}{2}
    \min_{1 \leq i < k}
    \norm*{ \nabla f (\bar x_i) }
    \leq
    \Lest \sqrt{\frac{S_{k-1}}{A_k}}
    + \sqrt{\frac{S_{k-1}}{32 A_k}} k (k-1) \Hest_{k-1}
    \leq
    \Lest \sqrt{\frac{S_{k-1}}{A_k}}
    \prn*{
      1 + \frac{3}{8 \sqrt{32}}
    },
  \end{alignat}
  where \revise{the last inequality uses} \cref{eq:continue_condition_k-1}.
  Using \cref{eq:Ak_lowerbound} and $1 + \frac{3}{8 \sqrt{32}} < \frac{2}{\sqrt{3}}$ concludes the proof.
\end{proof}

\subsection{Complexity bound}
Let $\Lmax$ denote the upper bound on the Lipschitz estimate $\Lest$ given in \cref{prop:Lest_upperbound}: $\Lmax \coloneqq \max \set{\Linit, \alpha \Ltrue}$.
The following theorem shows iteration complexity bounds for \cref{alg:proposed}.
Recall that $\alpha > 1$ and $0 < \beta \leq 1$ are the input parameters of \cref{alg:proposed}.
\begin{theorem}
  \label{thm:complexity}
  Suppose that \cref{asm:gradient_lip} holds \revise{and $\inf_{x \in \R^d} f(x) > - \infty$.}
  Let
  \begin{align}
    \Delta \coloneqq f(x_\mathrm{init}) - \inf_{x \in \R^d} f(x),\quad
    \revise{
    c_1 \coloneqq \log_\alpha \prn*{\frac{1}{\beta}},
    \quad\text{and}\quad
    c_2 \coloneqq 1 +  \log_\alpha \prn*{\frac{\Lmax}{\Linit}}.
    }
    \label{eq:def_c1_c2}
  \end{align}
  In \cref{alg:proposed}, when $\norm*{\nabla f(\bar x_k)} \leq \epsilon$ holds for the first time, the total iteration count $K$ is at most
  \begin{align}
    \inf_{\nu \in [0, 1]}
    \set*{
      91
      (1 + \sqrt{\revise{c_1}})
      \Delta
      \sqrt{\Lmax}
      \Htrue{\nu}^{\frac{1}{2 + 2 \nu}}
      \epsilon^{- \frac{4 + 3 \nu}{2 + 2 \nu}}
      +
      256 \revise{c_1}
      \Delta
      \Htrue{\nu}^{\frac{1}{1 + \nu}}
      \epsilon^{- \frac{2 + \nu}{1 + \nu}}
    }
    \revise{
      + 6 \sqrt{c_2 \Delta \Lmax} \epsilon^{-1}
      + c_2
    }.
  \end{align}
  In particular, if we set $\beta = 1$, then $\revise{c_1} = 0$ and the upper bound simplifies to
  \begin{align}
    \inf_{\nu \in [0, 1]}
    \set*{
      91
      \Delta
      \sqrt{\Lmax}
      \Htrue{\nu}^{\frac{1}{2 + 2 \nu}}
      \epsilon^{- \frac{4 + 3 \nu}{2 + 2 \nu}}
    }
    \revise{
      + 6 \sqrt{c_2 \Delta \Lmax} \epsilon^{-1}
      + c_2
    }.
    \label{eq:complexity_simple}
  \end{align}
\end{theorem}
\begin{proof}
  We classify the epochs into three types:
  \begin{itemize}
    \item 
    successful epoch: an epoch that does not find an $\epsilon$-stationary point and ends at \cref{alg-line-hb:restart_successful} with the descent condition~\cref{eq:armijo_rule} satisfied,
    \item
    unsuccessful epoch: an epoch that does not find an $\epsilon$-stationary point and ends at \cref{alg-line-hb:restart_unsuccessful} with the descent condition~\cref{eq:armijo_rule} unsatisfied,
    \item
    last epoch: the epoch that finds an $\epsilon$-stationary point.
  \end{itemize}
  Let $N_{\mathrm{suc}}$ and $N_{\mathrm{unsuc}}$ be the number of successful and unsuccessful epochs, respectively.
  Let $K_{\mathrm{suc}}$ be the total iteration number of all successful epochs.
  Below, we fix $\nu \in [0, 1]$ arbitrarily \revise{such that $\Htrue{\nu} < + \infty$}.
  \revise{(Note that there exists such a $\nu$ since $\Htrue{0} \leq 2 \Ltrue < + \infty$.)}

  \paragraph{Successful epochs.}
  Let us focus on a successful epoch \revise{and let $k$ denote the total number of iterations of the epoch we are focusing on, i.e., the epoch ends at iteration $k$.}
  \revise{
  We then have
  \begin{align}
    S_k
    \geq
    \frac{\epsilon^2 k^3}{8 \Lest^2}
    \label{eq:Sk_lowerbound2}
  \end{align}
  as follows: if $k = 1$, we have $S_k = \norm*{v_1}^2 = \frac{1}{\Lest^2} \norm*{\nabla f(x_0)}^2 > \frac{\epsilon^2}{\Lest^2} \geq \frac{\epsilon^2 k^3}{8 \Lest^2}$; if $k \geq 2$, \cref{lem:min_grad_xbar_upperbound2} gives $\epsilon < \Lest \sqrt{8 S_{k-1} / k^3} \leq \Lest \sqrt{8 S_k / k^3}$.%
  }
  On the other hand, putting the restart condition~\cref{eq:condition_restart} together with \cref{prop:Hest_upperbound} yields
  \begin{align}
    \frac{1}{4} \Lest
    < \frac{3}{8} \Lest
    < k (k+1) \Hest_k
    \leq 2 k^2 \Hest_k
    \leq 2 k^2 \Htrue{\nu} (k S_k)^{\frac{\nu}{2}}
  \end{align}
  and hence
  \begin{align}
    S_k
    \geq
    \frac{1}{k} \prn*{ \frac{\Lest}{8 k^2 \Htrue{\nu}} }^{2 / \nu}.
    \label{eq:Sk_lowerbound}
  \end{align}
  Combining \cref{eq:Sk_lowerbound,eq:Sk_lowerbound2} leads to
  \begin{align}
    S_k
    &=
    S_k^{\frac{2 + \nu}{2 + 2 \nu}}
    S_k^{\frac{\nu}{2 + 2 \nu}}
    \geq
    \prn*{\frac{\epsilon^2 k^3}{8 \Lest^2}}^{\frac{2 + \nu}{2 + 2 \nu}}
    \prn[\bigg]{ \frac{1}{k} \prn*{ \frac{\Lest}{8 k^2 \Htrue{\nu}}}^{2 / \nu} }^{\frac{\nu}{2 + 2 \nu}}
    =
    2^{- \frac{12 + 3 \nu}{2 + 2 \nu}}
    \Htrue{\nu}^{- \frac{1}{1 + \nu}}
    \epsilon^{\frac{2 + \nu}{1 + \nu}}
    \frac{k}{\Lest},\\
    S_k
    &=
    S_k^{\frac{4 + 3 \nu}{4 + 4 \nu}}
    S_k^{\frac{\nu}{4 + 4 \nu}}
    \geq
    \prn*{\frac{\epsilon^2 k^3}{8 \Lest^2}}^{\frac{4 + 3 \nu}{4 + 4 \nu}}
    \prn[\bigg]{ \frac{1}{k} \prn*{ \frac{\Lest}{8 k^2 \Htrue{\nu}}}^{2 / \nu} }^{\frac{\nu}{4 + 4 \nu}}
    =
    2^{- \frac{18 + 9 \nu}{4 + 4 \nu}}
    \Htrue{\nu}^{- \frac{1}{2 + 2 \nu}}
    \epsilon^{\frac{4 + 3 \nu}{2 + 2 \nu}}
    \frac{k^2}{\Lest^{3/2}}.
  \end{align}
  Plugging them into the \cref{eq:decrease_epoch_successful} yields
  \begin{align}
    f(x_0) - f(\xbest_k)
    \geq
    \frac{\Lest S_k}{4 k}
    &\geq
    2^{- \frac{16 + 7 \nu}{2 + 2 \nu}}
    \Htrue{\nu}^{- \frac{1}{1 + \nu}}
    \epsilon^{\frac{2 + \nu}{1 + \nu}}
    \geq
    2^{-8}
    \Htrue{\nu}^{- \frac{1}{1 + \nu}}
    \epsilon^{\frac{2 + \nu}{1 + \nu}},\\
    f(x_0) - f(\xbest_k)
    \geq
    \frac{\Lest S_k}{4 k}
    &\geq
    2^{- \frac{26 + 17 \nu}{4 + 4 \nu}}
    \Htrue{\nu}^{- \frac{1}{2 + 2 \nu}}
    \epsilon^{\frac{4 + 3 \nu}{2 + 2 \nu}}
    \frac{k}{\sqrt{\Lest}}
    \geq
    2^{-\frac{13}{2}}
    \Htrue{\nu}^{- \frac{1}{2 + 2 \nu}}
    \epsilon^{\frac{4 + 3 \nu}{2 + 2 \nu}}
    \frac{k}{\sqrt{\Lmax}}
  \end{align}
  since $\nu \geq 0$.
  Summing these bounds over all successful epochs results in
  \begin{align}
    \Delta
    \geq
    2^{-8}
    \Htrue{\nu}^{- \frac{1}{1 + \nu}}
    \epsilon^{\frac{2 + \nu}{1 + \nu}}
    N_{\mathrm{suc}},\quad
    \Delta
    \geq
    2^{-\frac{13}{2}}
    \Htrue{\nu}^{- \frac{1}{2 + 2 \nu}}
    \epsilon^{\frac{4 + 3 \nu}{2 + 2 \nu}}
    \frac{K_{\mathrm{suc}}}{\sqrt{\Lmax}},
  \end{align}
  and hence
  \begin{align}
    N_{\mathrm{suc}}
    \leq
    2^8
    \Delta
    \Htrue{\nu}^{\frac{1}{1 + \nu}}
    \epsilon^{- \frac{2 + \nu}{1 + \nu}},\quad
    K_{\mathrm{suc}}
    \leq
    2^{\frac{13}{2}}
    \Delta
    \sqrt{\Lmax}
    \Htrue{\nu}^{\frac{1}{2 + 2 \nu}}
    \epsilon^{- \frac{4 + 3 \nu}{2 + 2 \nu}}.
    \label{eq:Nsuc_Ksuc_upperbound}
  \end{align}

  \paragraph{Other epochs.}
  Let $k_1,\dots,k_{N_{\mathrm{unsuc}}}$ and $k_{N_{\mathrm{unsuc}} + 1}$ be the iteration number of unsuccessful and last epochs, respectively.
  Then, the total iteration number of the epochs can be bounded with the Cauchy--Schwarz inequality as follows:
  \begin{align}
    \sum_{i=1}^{N_{\mathrm{unsuc}} + 1} k_i
    &=
    \reviset{
    \sum_{i:\, k_i = 1} k_i
    + \sum_{i:\, k_i \geq 2} k_i
    }\\
    &\leq
    N_{\mathrm{unsuc}} + 1 + \sum_{i:\, k_i \geq 2} k_i
    \leq
    N_{\mathrm{unsuc}} + 1
    + \sqrt{N_{\mathrm{unsuc}} + 1} \sqrt{ \sum_{i:\, k_i \geq 2} k_i^2 },
    \label{eq:Kunsuclast_upperbound}
  \end{align}
  \revise{where $\sum_{i:\, k_i \geq 2}$ denotes a sum over $i = 1, \dots, N_{\mathrm{unsuc}} + 1$ such that $k_i \geq 2$.}
  We will evaluate $N_{\mathrm{unsuc}}$ and the sum of $k_i^2$.
  First, we have $\Linit \beta^{N_{\mathrm{suc}}} \alpha^{N_{\mathrm{unsuc}}} \leq \Lmax$ and hence
  \begin{align}
    N_{\mathrm{unsuc}}
    \leq
    \revise{
    c_1 N_{\mathrm{suc}}
    + c_2 - 1
    \leq
    2^8 c_1
    \Delta
    \Htrue{\nu}^{\frac{1}{1 + \nu}}
    \epsilon^{- \frac{2 + \nu}{1 + \nu}}
    + c_2 - 1
    }
    \label{eq:Nunsuc_upperbound}
  \end{align}
  from \cref{eq:Nsuc_Ksuc_upperbound}, \revise{where $c_1$ and $c_2$ are defined by \cref{eq:def_c1_c2}.} 
  Next, let us focus on an epoch that ends at iteration $k \geq 2$.
  \cref{lem:min_grad_xbar_upperbound2} gives $\epsilon < \Lest \sqrt{8 S_{k-1} / k^3}$ and hence $S_{k-1} \geq \frac{\epsilon^2 k^3}{8 \Lest^2}$.
  Plugging this bound into \cref{eq:decrease_epoch_unsuccessful} yields
  \begin{align}
    f(x_0) - f(\xbest_k)
    \geq
    \frac{\Lest S_{k-1}}{4k}
    \geq
    \frac{\epsilon^2 k^2}{2^5 \Lest}.
  \end{align}
  Summing this bound over all unsuccessful and last epochs results in
  \begin{align}
    \sum_{i:\, k_i \geq 2} k_i^2
    \leq 
    \frac{2^5 \Delta \Lmax}{\epsilon^2}.
    \label{eq:ki_squaredsum_upperbound}
  \end{align}
  Plugging \cref{eq:Nunsuc_upperbound,eq:ki_squaredsum_upperbound} into \cref{eq:Kunsuclast_upperbound} yields
  \begin{align}
    \sum_{i=1}^{N_{\mathrm{unsuc}} + 1} k_i
    &\leq
    2^8 \revise{c_1}
    \Delta
    \Htrue{\nu}^{\frac{1}{1 + \nu}}
    \epsilon^{- \frac{2 + \nu}{1 + \nu}}
    + \revise{c_2}
    +
    \sqrt{
      2^8 \revise{c_1}
      \Delta
      \Htrue{\nu}^{\frac{1}{1 + \nu}}
      \epsilon^{- \frac{2 + \nu}{1 + \nu}}
      + \revise{c_2}
    }
    \sqrt{
      \frac{2^5 \Delta \Lmax}{\epsilon^2}
    }\\
    &\leq
    2^8 \revise{c_1}
    \Delta
    \Htrue{\nu}^{\frac{1}{1 + \nu}}
    \epsilon^{- \frac{2 + \nu}{1 + \nu}}
    + \revise{c_2}
    +
    2^{\frac{13}{2}}
    \sqrt{\revise{c_1}}
    \Delta
    \sqrt{\Lmax}
    \Htrue{\nu}^{\frac{1}{2 + 2 \nu}}
    \epsilon^{- \frac{4 + 3 \nu}{2 + 2 \nu}}
    \revise{+ 2^{\frac{5}{2}} \sqrt{c_2 \Delta \Lmax} \epsilon^{-1}},
  \end{align}
  \revise{where the last inequality uses $\sqrt{a + b} \leq \sqrt{a} + \sqrt{b}$ for $a, b \geq 0$.}
  Putting this bound together with \cref{eq:Nsuc_Ksuc_upperbound} gives an upper bound on the total iteration number of all epochs:
  \begin{align}
    K_{\mathrm{suc}}
    + \sum_{i=1}^{N_{\mathrm{unsuc}} + 1} k_i
    \leq
    91
    (1 + \sqrt{\revise{c_1}})
    \Delta
    \sqrt{\Lmax}
    \Htrue{\nu}^{\frac{1}{2 + 2 \nu}}
    \epsilon^{- \frac{4 + 3 \nu}{2 + 2 \nu}}
    +
    256 \revise{c_1}
    \Delta
    \Htrue{\nu}^{\frac{1}{1 + \nu}}
    \epsilon^{- \frac{2 + \nu}{1 + \nu}}
    \revise{
      + 6 \sqrt{c_2 \Delta \Lmax} \epsilon^{-1}
      + c_2
    },
  \end{align}
  where we have used $2^{\frac{13}{2}} < 91$, $2^8 = 256$, \revise{and $2^{\frac{5}{2}} < 6$}.
  Since $\nu \in [0, 1]$ is now arbitrary, taking the infimum completes the proof.
\end{proof}

\cref{alg:proposed} evaluates the objective function and its gradient at two points, $x_k$ and $\bar x_k$, in each iteration.
Therefore, the number of evaluations is of the same order as the iteration complexity in \cref{thm:complexity}.

The complexity bounds given in \cref{thm:complexity} may look somewhat unfamiliar since they involve an $\inf$-operation on $\nu$.
Such a bound is a significant benefit of $\nu$-independent algorithms.
The $\nu$-dependent prototype algorithm described immediately after \cref{lem:two_inequalities_for_analysis_holder} achieves the bound
\begin{align}
  91
  (1 + \sqrt{\revise{c_1}})
  \Delta
  \sqrt{\Lmax}
  \Htrue{\nu}^{\frac{1}{2 + 2 \nu}}
  \epsilon^{- \frac{4 + 3 \nu}{2 + 2 \nu}}
  +
  256 \revise{c_1}
  \Delta
  \Htrue{\nu}^{\frac{1}{1 + \nu}}
  \epsilon^{- \frac{2 + \nu}{1 + \nu}}
  \revise{
    + 6 \sqrt{c_2 \Delta \Lmax} \epsilon^{-1}
    + c_2
  },
\end{align}
only for the given $\nu$.
In contrast, \cref{alg:proposed} is $\nu$-independent and automatically achieves the bound with the optimal $\nu$, as shown in \cref{thm:complexity}.
The fact that the optimal $\nu$ is difficult to find also points to the advantage of our $\nu$-independent algorithm.

The complexity bound~\cref{eq:complexity_simple} also gives a looser bound:
\begin{align}
  \inf_{\nu \in [0, 1]}
  \set*{
    91
    \Delta
    \sqrt{\Lmax}
    \Htrue{\nu}^{\frac{1}{2 + 2 \nu}}
    \epsilon^{- \frac{4 + 3 \nu}{2 + 2 \nu}}
  }
  + O(\epsilon^{-1})
  \leq
  91
  \Delta
  \sqrt{\Lmax \Htrue{0}}
  \epsilon^{-2}
  + O(\epsilon^{-1})
  \leq
  91 \sqrt{2}
  \Delta
  \Lmax
  \epsilon^{-2}
  + O(\epsilon^{-1}),
\end{align}
where we have taken $\nu = 0$ and have used $\Htrue{0} \leq 2 \Ltrue \leq 2 \Lmax$.
This bound matches the classical bound of $O(\epsilon^{-2})$ for GD.
\cref{thm:complexity} thus shows that our HB method has a more elaborate complexity bound than GD.

\revise{
\begin{remark}
  \label{rem:local_lip_holder}
  Although we employed global Lipschitz and H\"older continuity in \cref{asm:gradient_lip,def:holder_constant}, they can be restricted to the region where the iterates reach.
  More precisely, if we assume that the iterates $(x_k)$ generated by \cref{alg:proposed} are contained in some convex set $C \subseteq \R^d$, we can replace all $\R^d$ in our analysis with $C$; we can obtain the same complexity bound as \cref{thm:complexity} with Lipschitz and H\"older continuity on $C$.\footnote{
    \revise{
      We omit the proof, which is essentially the same, only replacing all $\R^d$ with $C$. The convexity of $C$ is necessary to guarantee that the averaged solution $\bar x_k$ also belongs to $C$.
    }
  }
\end{remark}
}

\section{Numerical experiments}
This section compares the performance of the proposed method with several existing algorithms.
The experimental setup, including the compared algorithms and problem instances, follows \citep{marumo2022parameter}.
We implemented \revise{the code} in Python with JAX~\citep{jax2018github} and Flax~\citep{flax2020github} and executed them on a computer with an Apple \revise{M3} Chip (\revise{12} cores) and \revise{36} GB RAM.
The source code used in the experiments is available on GitHub.%
\footnote{
  \url{https://github.com/n-marumo/restarted-hb}
}

\subsection{Compared algorithms}
We compared the following six algorithms.
\begin{itemize}
  \item 
  \Proposed is \cref{alg:proposed} with parameters set as $(\Linit, \alpha, \beta) = (10^{-3}, \allowbreak 2, 0.1)$.
  \item 
  \GD is a gradient descent method with Armijo-type backtracking.
  This method has input parameters $\Linit$, $\alpha$, and $\beta$ similar to those in \Proposed, which were set as $(\Linit, \alpha, \beta) = (10^{-3}, 2, 0.9)$.
  \item
  \JNJ \citep[Algorithm~2]{jin2018accelerated} is an accelerated gradient (AG) method for nonconvex optimization.
  The parameters were set in accordance with \citep[Eq.~(3)]{jin2018accelerated}.
  The equation involves constants $c$ and $\chi$, whose values are difficult to determine; we set them as $c = \chi = 1$.
  \item
  \LL \citep[Algorithm~2]{li2022restarted} is another AG method.
  The parameters were set in accordance with \citep[Theorem~2.2 and Section~4]{li2022restarted}.
  \item
  \MT \citep[Algorithm~1]{marumo2022parameter} is another AG method.
  The parameters were set in accordance with \citep[Section~6.1]{marumo2022parameter}.
  \revise{
  \item
  \LBFGS is the limited-memory BFGS method~\cite{byrd1995limited}.
  We used SciPy \cite{virtanen2020scipy} for the method, i.e., \texttt{scipy.optimize.minimize} with option \texttt{method="L-BFGS-B"}.
  }
\end{itemize}

The parameter setting for \JNJ and \LL requires the values of the Lipschitz constants $\Ltrue$ and $\Htrue{1}$ and the target accuracy $\epsilon$.
For these two methods, we tuned the best $\Ltrue$ among $\set{10^{-4},10^{-3},\dots,\reviset{10^{10}}}$ and set $\Htrue{1} = 1$ and $\epsilon = 10^{-16}$ following \citep{li2022restarted,marumo2022parameter}.
It should be noted that if these values deviate from the actual ones, the methods do not guarantee convergence.

\subsection{Problem instances}
We tested the algorithms on \revise{seven} different instances.
\revise{
The first four instances are benchmark functions from \citep{jamil2013literature}.
\begin{itemize}
  \item
  Dixon--Price function \citep{dixon1989truncated}:
  \begin{align}
    \min_{(x_1,\dots,x_d) \in \R^d}\ 
    (x_1 - 1)^2 + \sum_{i=2}^d i (2 x_i^2 - x_{i-1})^2.
    \label{eq:dixon_price}
  \end{align}
  The optimum is $f(x^*) = 0$ at $x^*_i = 2^{2^{1-i} - 1}$ for $1 \leq i \leq d$.
  \item
  Powell function \citep{powell1962iterative}:
  \begin{align}
    \min_{(x_1,\dots,x_d) \in \R^d}\ 
    \sum_{i=1}^{\floor{d/4}} \prn*{
      \prn*{x_{4i-3} + 10 x_{4i-2}}^2
      + 5 \prn*{x_{4i-1} - x_{4i}}^2
      + \prn*{x_{4i-2} - 2 x_{4i-1}}^4
      + 10 \prn*{x_{4i-3} - x_{4i}}^4
    }.
    \label{eq:powell}
  \end{align}
  The optimum is $f(x^*) = 0$ at $x^* = (0, \dots, 0)$.
  \item
  Qing Function \citep{qing2006dynamic}:
  \begin{align}
    \min_{(x_1,\dots,x_d) \in \R^d}\ 
    \sum_{i=1}^{d-1} (x_i^2 - i)^2.
    \label{eq:qing}
  \end{align}
  The optimum is $f(x^*) = 0$ at $x^* = (\pm \sqrt{1}, \pm \sqrt{2}, \dots, \pm \sqrt{d})$.
  \item
  Rosenbrock function \citep{rosenbrock1960automatic}:
  \begin{align}
    \min_{(x_1,\dots,x_d) \in \R^d}\ 
    \sum_{i=1}^{d-1} \prn*{
      100 \prn*{x_{i+1} - x_i^2}^2
      + (x_i - 1)^2
    }.
    \label{eq:rosenbrock}
  \end{align}
  The optimum is $f(x^*) = 0$ at $x^* = (1, \dots, 1)$.
\end{itemize}
The dimension $d$ of the above problems was fixed as $d = 10^6$.
The starting point was set as $\xinit = x^* + \delta$, where $x^*$ is the optimal solution, and each entry of $\delta$ was drawn from the normal distribution $\mathcal N(0, 1)$.
For the Qing function~\cref{eq:qing}, we used $x^* = (\sqrt{1}, \sqrt{2}, \dots, \sqrt{d})$ to set the starting point.
}

\revise{
  The other three instances are more practical examples from machine learning.
}
\begin{itemize}
  \item 
  Training a neural network for classification with the MNIST dataset:
  \begin{align}
    \min_{w \in \R^d}\ 
    &
    \frac{1}{N}
    \sum_{i=1}^N
    \ell_{\mathrm{CE}}(y_i, \phi_1(x_i; w)).
    \label{eq:exp_classification}
  \end{align}
  The vectors $x_1,\dots,x_N \in \R^M$ and $y_1,\dots,y_N \in \set{0, 1}^K$ are given data, $\ell_{\mathrm{CE}}$ is the cross-entropy loss, and $\phi_1(\cdot; w): \R^M \to \R^K$ is a neural network parameterized by $w \in \R^d$.
  \revise{
    We used a three-layer fully connected network with bias parameters.
    The layers each have $M$, $32$, $16$, and $K$ nodes, where $M = 784$ and $K = 10$.
    The hidden layers have the logistic sigmoid activation, and the output layer has the softmax activation.
    The total number of the parameters is $d = (784 \times 32 + 32 \times 16 + 16 \times 10) + (32 + 16 + 10) = 25818$.%
  }
  The data size is $N = 10000$.
  \item
  Training an autoencoder for the MNIST dataset:
  \begin{align}
    \min_{w \in \R^d}\ 
    &
    \frac{1}{2MN}
    \sum_{i=1}^N
    \norm*{x_i - \phi_2(x_i; w)}^2.
    \label{eq:exp_ae}
  \end{align}
  The vectors $x_1,\dots,x_N \in \R^M$ are given data, and $\phi_2(\cdot; w): \R^M \to \R^M$ is a neural network parameterized by $w \in \R^d$.
  \revise{
    We used a four-layer fully connected network with bias parameters.
    The layers each have $M$, $32$, $16$, $32$, and $M$ nodes, where $M = 784$.
    The hidden and output layers have the logistic sigmoid activation.
    The total number of the parameters is $d = (784 \times 32 + 32 \times 16 + 16 \times 32 + 32 \times 784) + (32 + 16 + 32 + 784) = 52064$.%
  }
  The data size is $N = 10000$.
  \item
  Low-rank matrix completion with the MovieLens-100K dataset:
  \begin{align}
    \min_{\substack{U \in \R^{p \times r}\\V \in \R^{q \times r}}}\ 
    &
    \frac{1}{2 N}
    \sum_{(i, j, s) \in \Omega}
    \prn*{(U V^\top)_{ij} - s}^2
    + \frac{1}{2 N} \norm*{U^\top U - V^\top V}_{\mathrm{F}}^2.
    \label{eq:exp_mf}
  \end{align}
  The set $\Omega$ consists of $N = 100000$ observed entries of a $p \times q$ data matrix, and $(i, j, s) \in \Omega$ means that the $(i, j)$-th entry is $s$.
  The second term with the Frobenius norm $\norm{\cdot}_{\mathrm{F}}$ was proposed in \citep{tu2016low} as a way to balance $U$ and $V$.
  The size of the data matrix is $p = 943$ times $q = 1682$, and we set the rank as $r \in \set{100, 200}$.
  Thus, the number of variables is $pr + qr \in \set{262500, 525000}$.
\end{itemize}


\revise{
  Although we did not check whether the above seven instances have globally Lipschitz continuous gradients or Hessians, we confirmed in our experiments that the iterates generated by each algorithm were bounded.
  Since all of the above instances are continuously thrice differentiable, both the gradients and Hessians are Lipschitz continuous in the bounded domain.
  Considering \cref{rem:local_lip_holder}, we can say that in the experiments, the proposed algorithm achieves the same complexity bound as \cref{thm:complexity}. 
}

\subsection{Results}
\revise{
  \Cref{fig:experiments_benchmark} illustrates the results with the four benchmark functions.\footnote{
    \revise{
      To obtain results of \LBFGS, we ran the SciPy functions multiple times with the maximum number of iterations set to $2^0, 2^1, 2^2,\dots$ because we cannot obtain the solution at each iteration while running SciPy codes of \LBFGS, but only the final result.
      The results are thus plotted as markers instead of lines in \cref{fig:experiments_benchmark,fig:experiments_benchmark_time,fig:experiments}.
    }
  }
  The horizontal axis is the number of calls to the oracle that computes both $f(x)$ and $\nabla f(x)$ at a given point $x \in \R^d$.

  Let us first focus on the methods other than \LBFGS, which is very practical but does not have complexity guarantees for general nonconvex functions, unlike the other methods.
  \Cref{fig:exp_dixonprice,fig:exp_powell} show that \Proposed converged faster than the existing methods except for \LBFGS, and \cref{fig:exp_qing} shows that \Proposed and \MT converged fast.
  \Cref{fig:exp_rosenbrock} shows that \GD and \LL attained a small objective function value, while \GD and \Proposed converged fast regarding gradient norm.
  In summery, the proposed algorithm was stable and fast.

  \LBFGS successfully solved the four benchmarks, but we should note that the results do not imply that \LBFGS converged faster than the proposed algorithm in terms of execution time.
  \Cref{fig:experiments_benchmark_time} provides the four figures in the right column of \cref{fig:experiments_benchmark}, with the horizontal axis replaced by the elapsed time.
  \Cref{fig:experiments_benchmark_time} shows that \Proposed converged comparably or faster in terms of time than \LBFGS.
  One reason for the large difference in the apparent performance of \LBFGS in \Cref{fig:experiments_benchmark,fig:experiments_benchmark_time} is that the computational costs of the non-oracle parts in \LBFGS, such as updating the Hessian approximation and solving linear systems, are not negligible.
  In contrast, the proposed algorithm does not require heavy computation besides oracle calls and is more advantageous in execution time when function and gradient evaluations are low-cost.

  \Cref{fig:experiments} presents the results with the machine learning instances.
  Similar to \cref{fig:experiments_benchmark}, \Cref{fig:experiments} shows that the proposed algorithm performed comparably or better than the existing methods except for \LBFGS, especially in reducing the gradient.

  \Cref{fig:experiments_objlh} illustrates the objective function value $f(x_k)$ and the estimates $\Lest$ and $\Hest_k$ at each iteration of the proposed algorithm for the machine learning instances.
  The iterations at which a restart occurred are also marked; ``successful'' and ``unsuccessful'' mean restarts at Line \ref{alg-line-hb:restart_successful} and Line \ref{alg-line-hb:restart_unsuccessful} of \cref{alg:proposed}, respectively.
  This figure shows that the proposed algorithm restarts frequently in the early stages but that the frequency decreases as the iterations progress.
  The frequent restarts in the early stages help update the estimate $\Lest$; $\Lest$ reached suitable values in the first few iterations, even though it was initialized to a pretty small value, $\Linit = 10^{-3}$.
  The infrequent restarts in later stages enable the algorithm to take full advantage of the HB momentum.


}




\begin{figure}
  \centering
  \includegraphics[height=3.5ex]{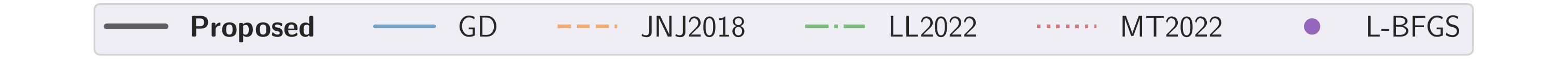}\par\medskip%
  \subfloat[Dixon--Price function \cref{eq:dixon_price}\label{fig:exp_dixonprice}]{%
    \includegraphics[width=0.43\linewidth]{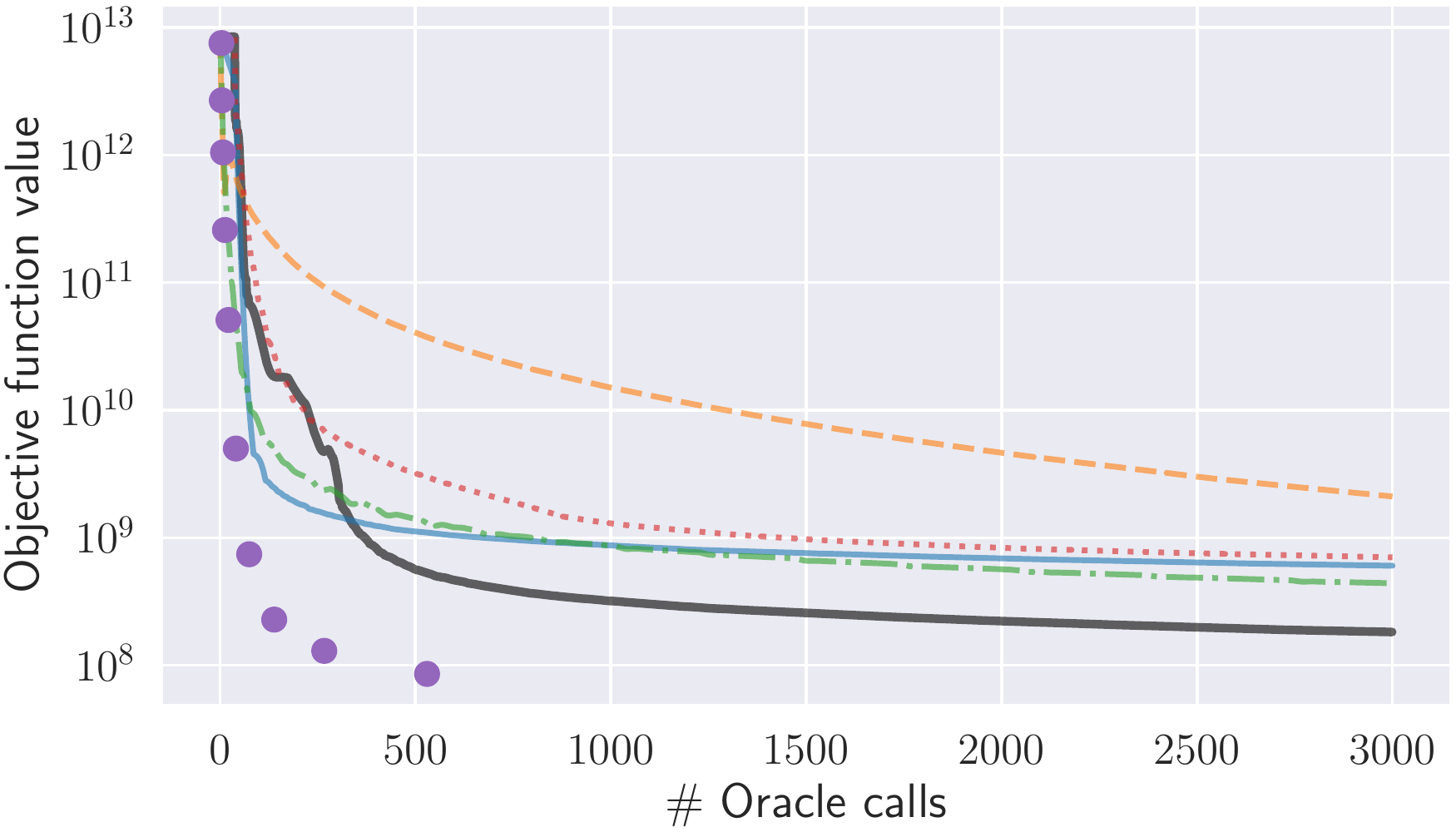}\hspace{0.1\linewidth}%
    \includegraphics[width=0.43\linewidth]{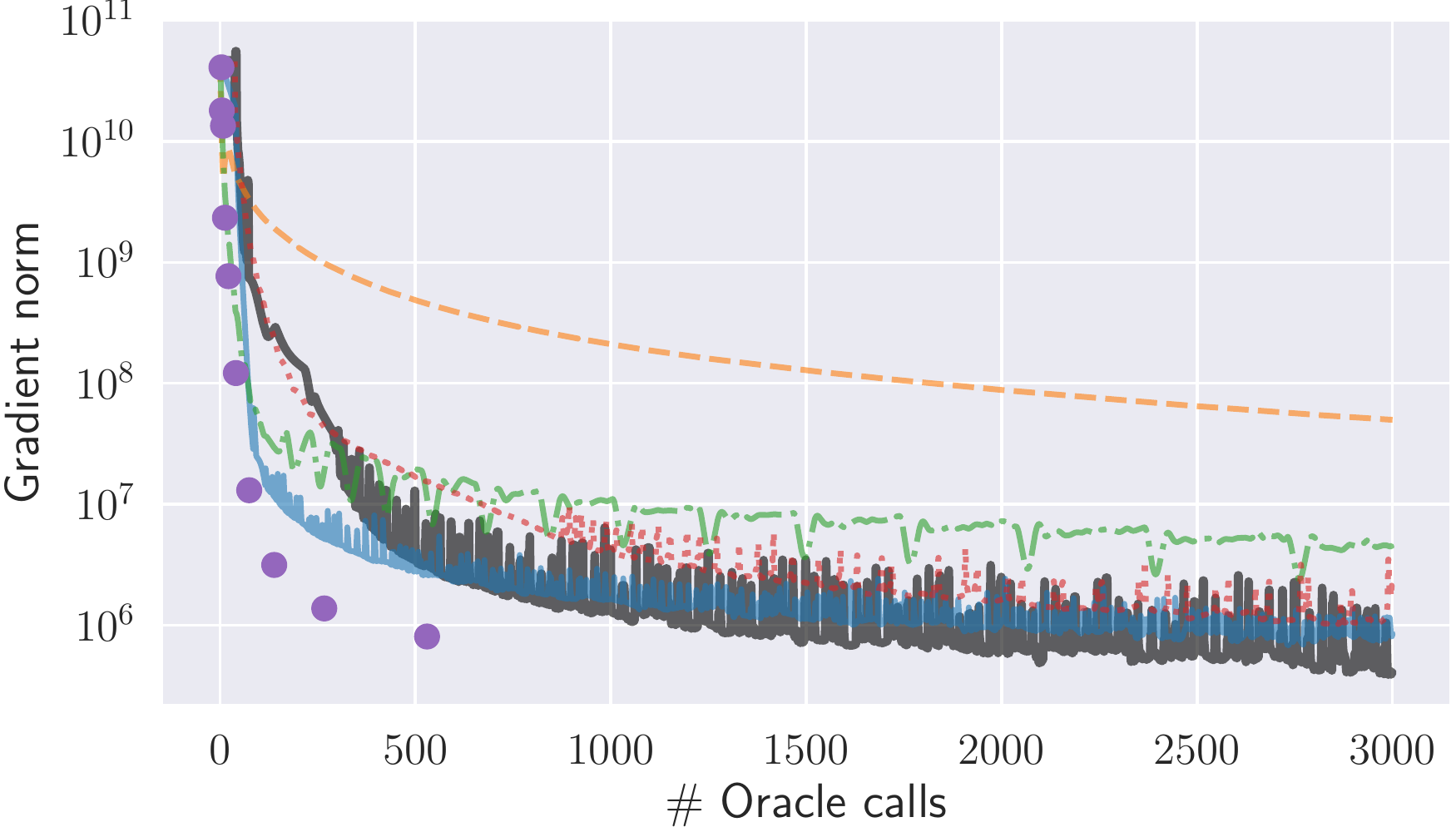}%
  }\par\medskip%
  \subfloat[Powell function \cref{eq:powell}\label{fig:exp_powell}]{%
    \includegraphics[width=0.43\linewidth]{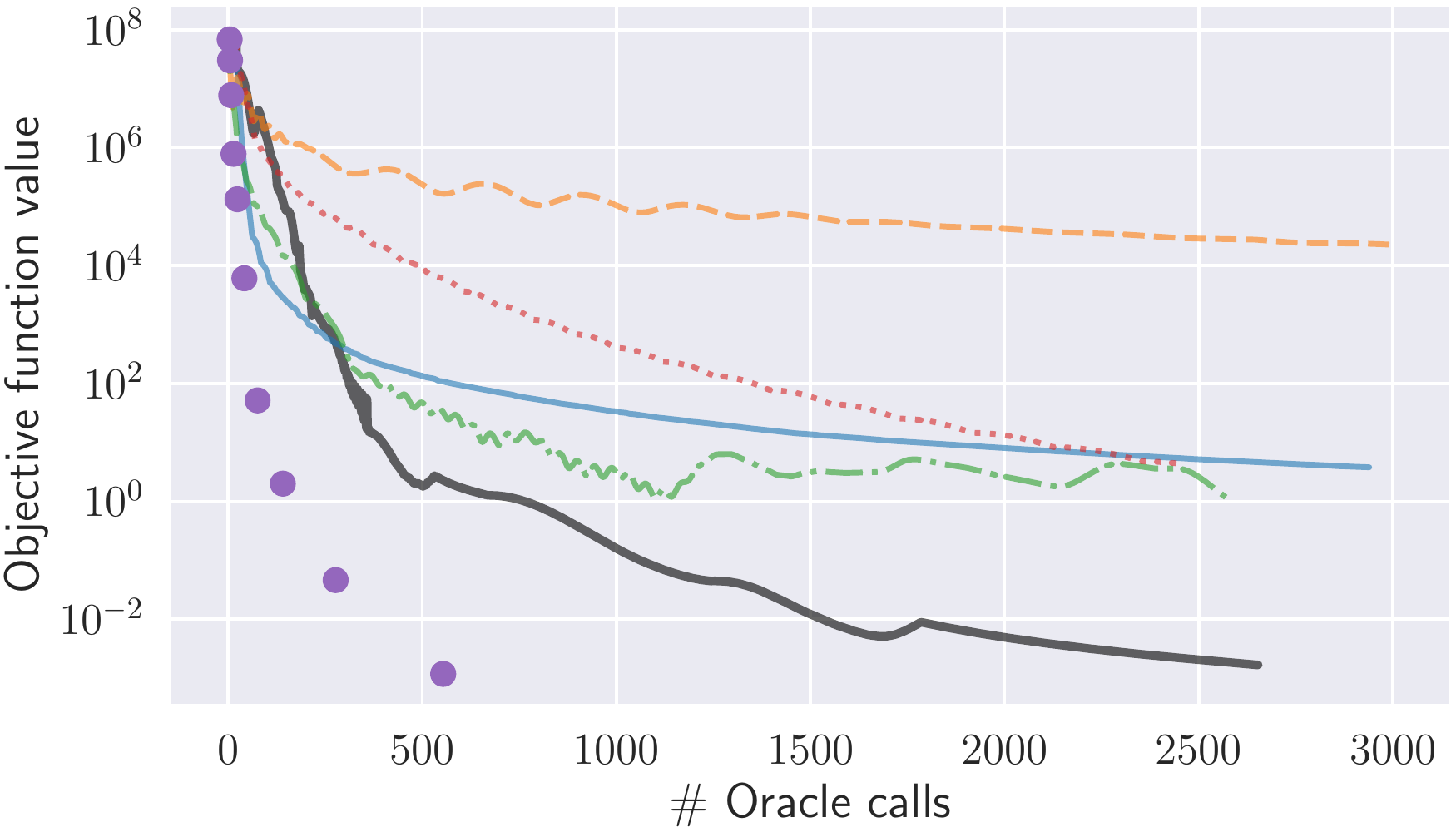}\hspace{0.1\linewidth}%
    \includegraphics[width=0.43\linewidth]{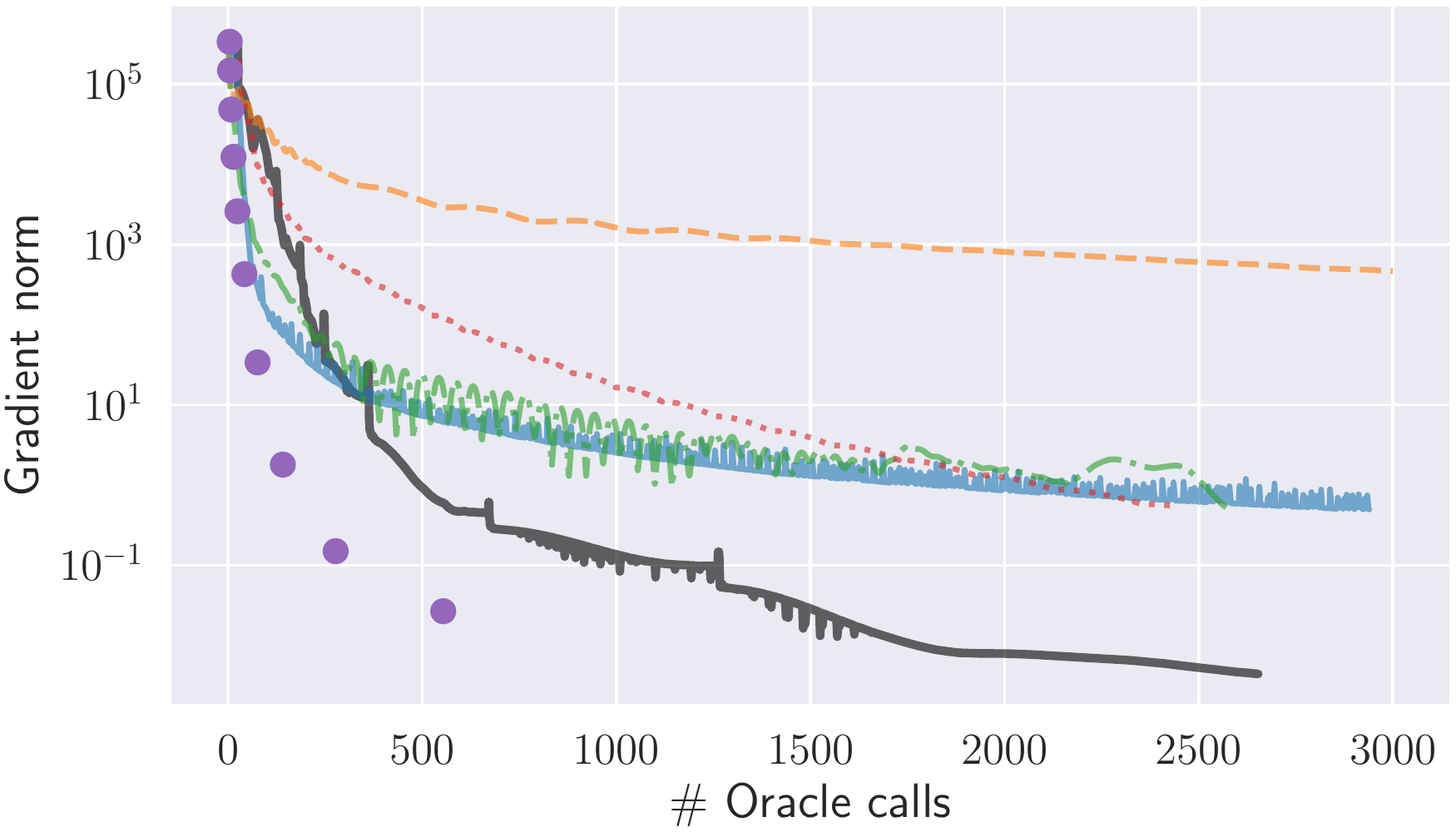}%
  }\par\medskip%
  \subfloat[Qing function \cref{eq:qing}\label{fig:exp_qing}]{%
    \includegraphics[width=0.43\linewidth]{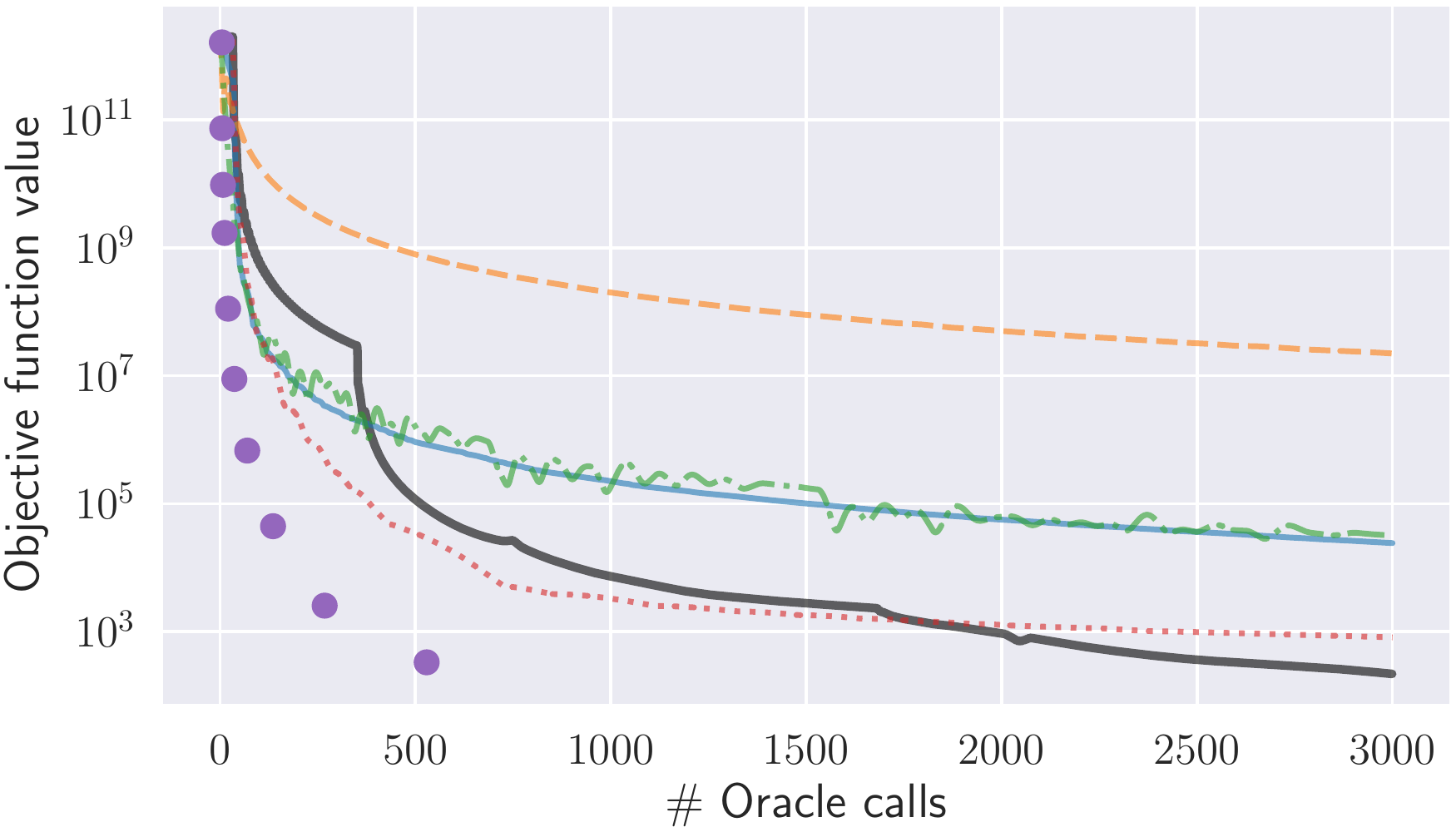}\hspace{0.1\linewidth}%
    \includegraphics[width=0.43\linewidth]{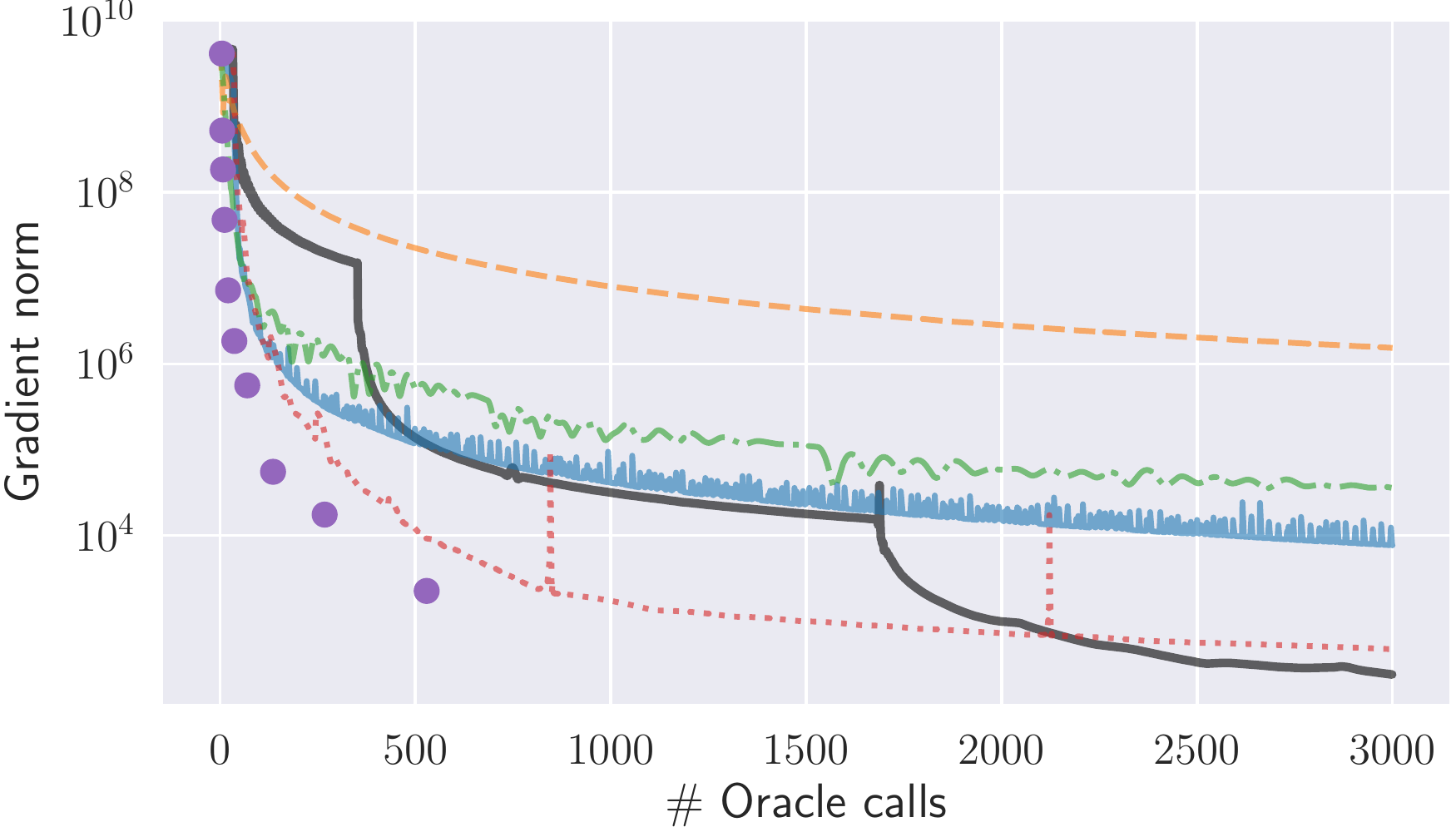}%
  }\par\medskip%
  \subfloat[Rosenbrock function \cref{eq:rosenbrock}\label{fig:exp_rosenbrock}]{%
    \includegraphics[width=0.43\linewidth]{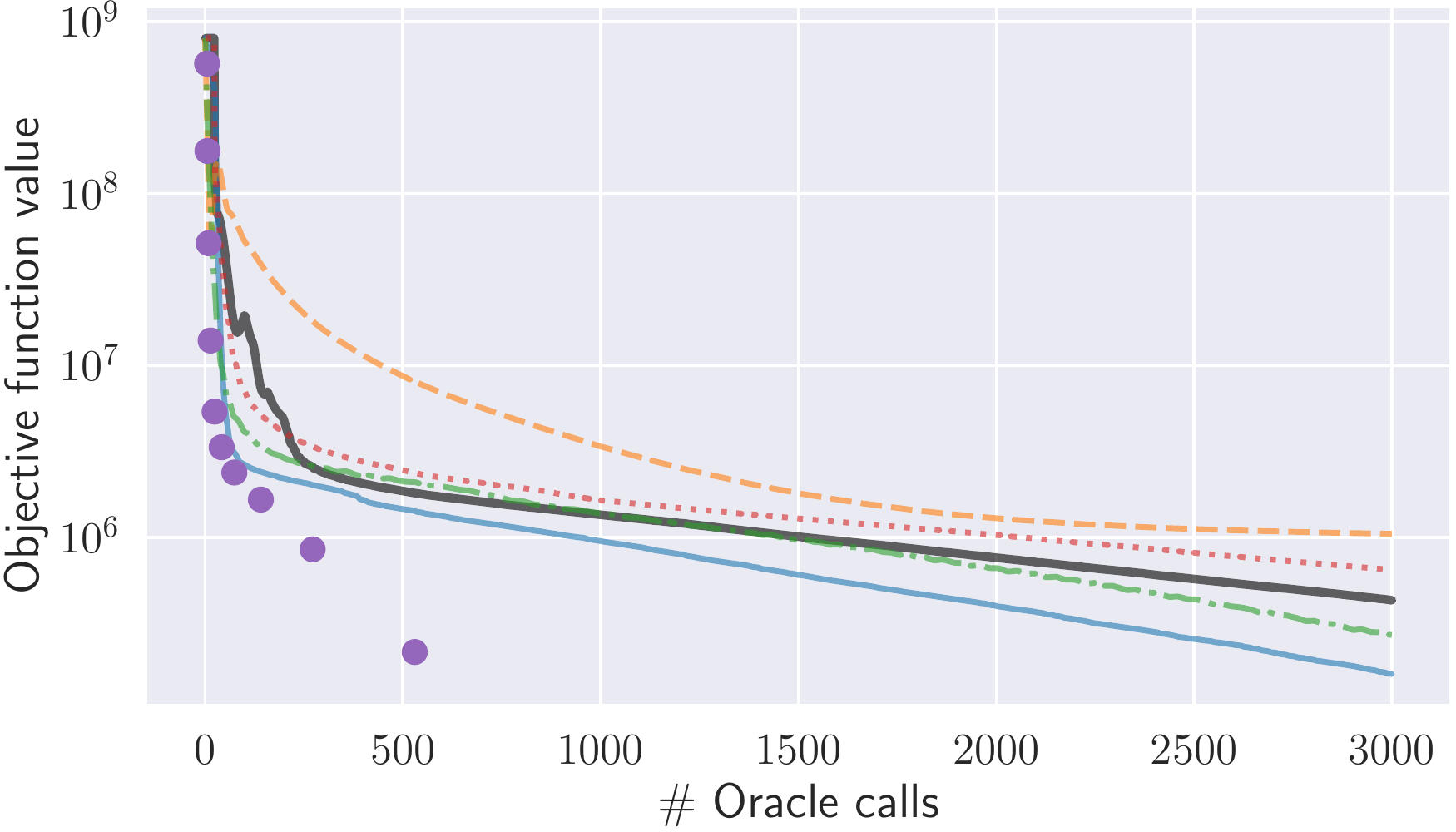}\hspace{0.1\linewidth}%
    \includegraphics[width=0.43\linewidth]{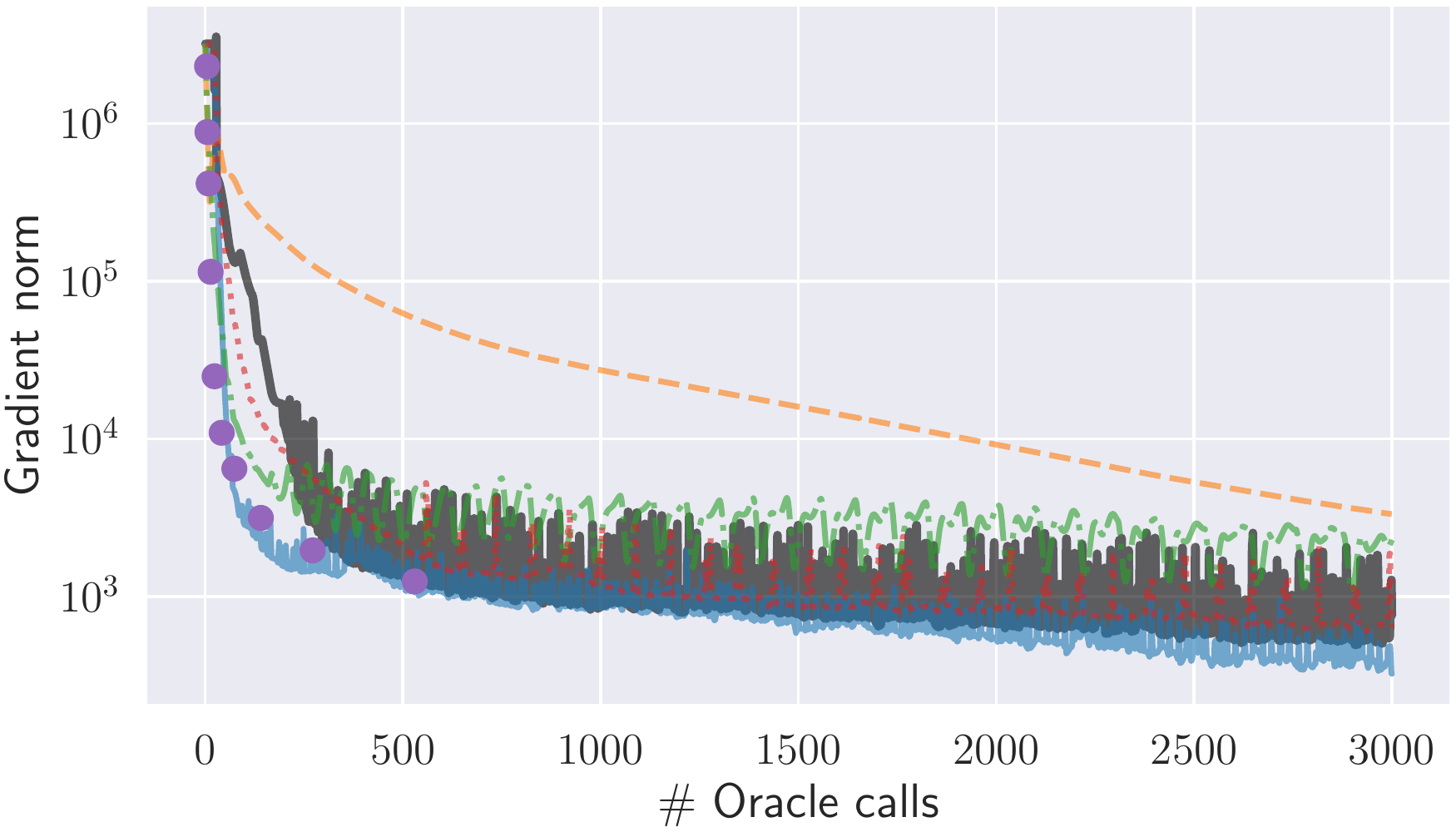}%
  }\par
  \caption{
    \revise{
      Numerical results with benchmark functions.\label{fig:experiments_benchmark}
    }
  }
\end{figure}

\begin{figure}[t]
  \centering
  \includegraphics[height=3.5ex]{fig/legend.pdf}\par\medskip%
  \subfloat[Dixon--Price function \cref{eq:dixon_price}\label{fig:exp_time_dixonprice}]{%
    \includegraphics[width=0.43\linewidth]{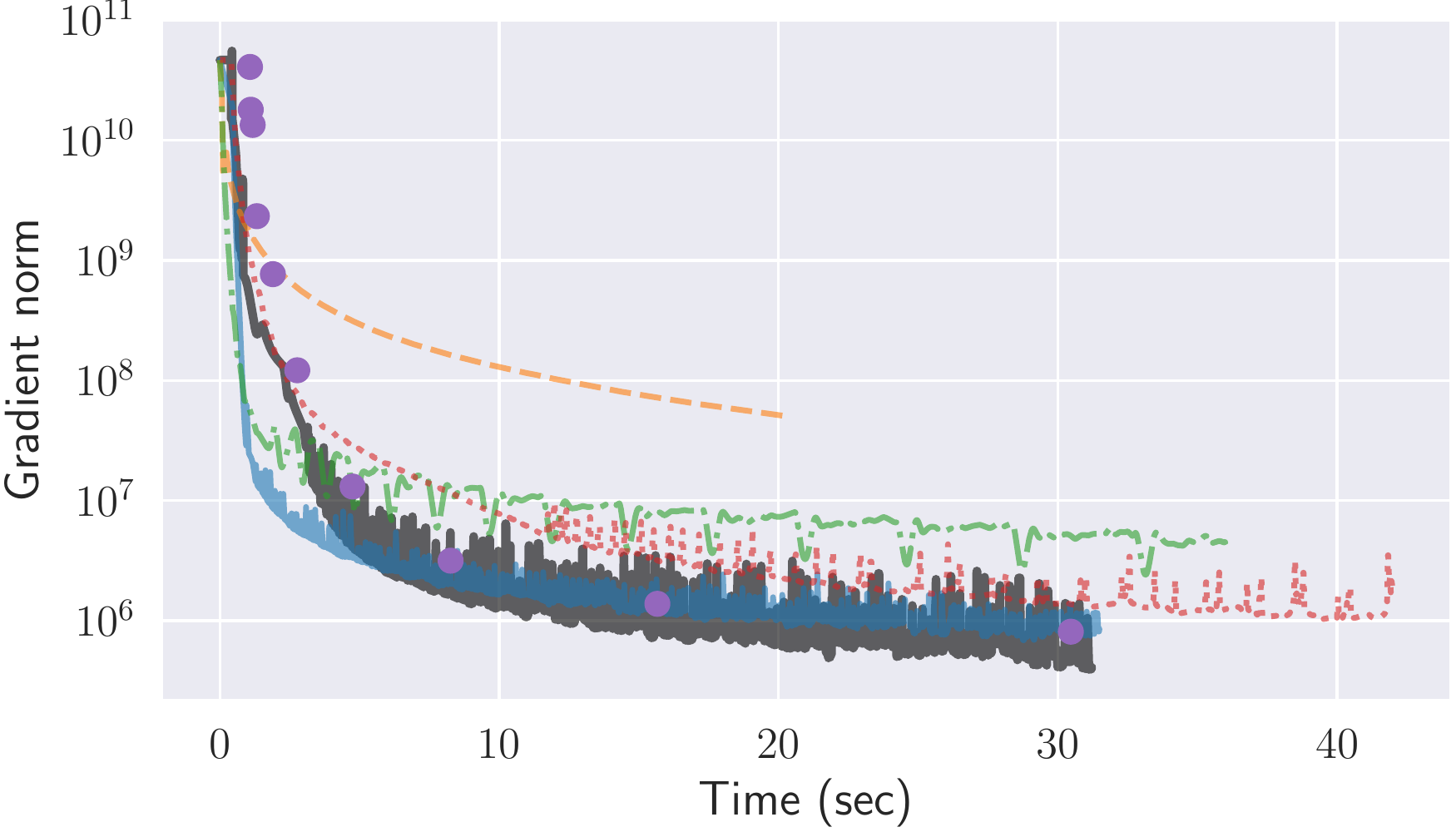}%
  }\hspace{0.1\linewidth}%
  \subfloat[Powell function \cref{eq:powell}\label{fig:exp_time_powell}]{%
    \includegraphics[width=0.43\linewidth]{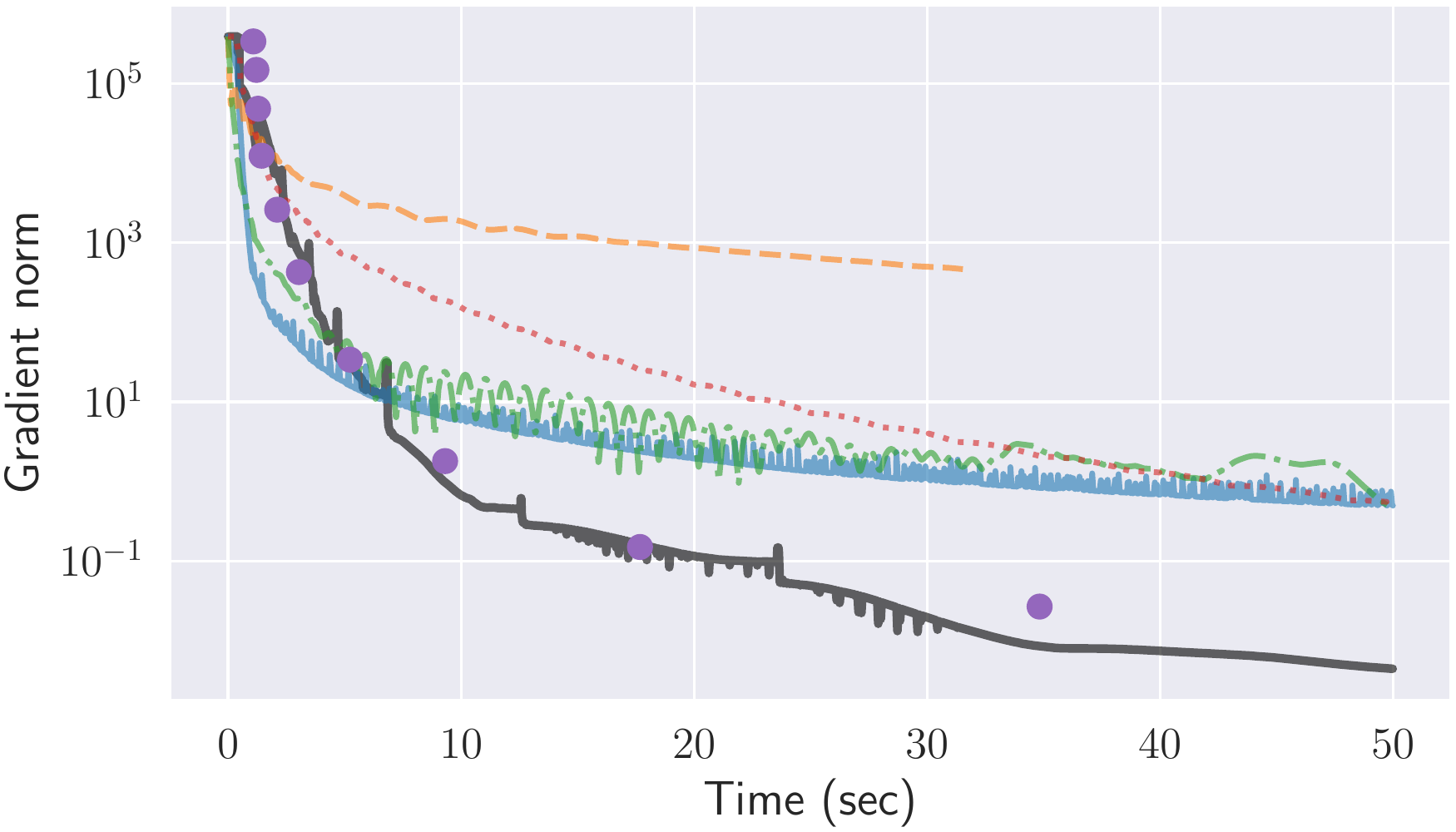}%
  }\par\medskip%
  \subfloat[Qing function \cref{eq:qing}\label{fig:exp_time_qing}]{%
    \includegraphics[width=0.43\linewidth]{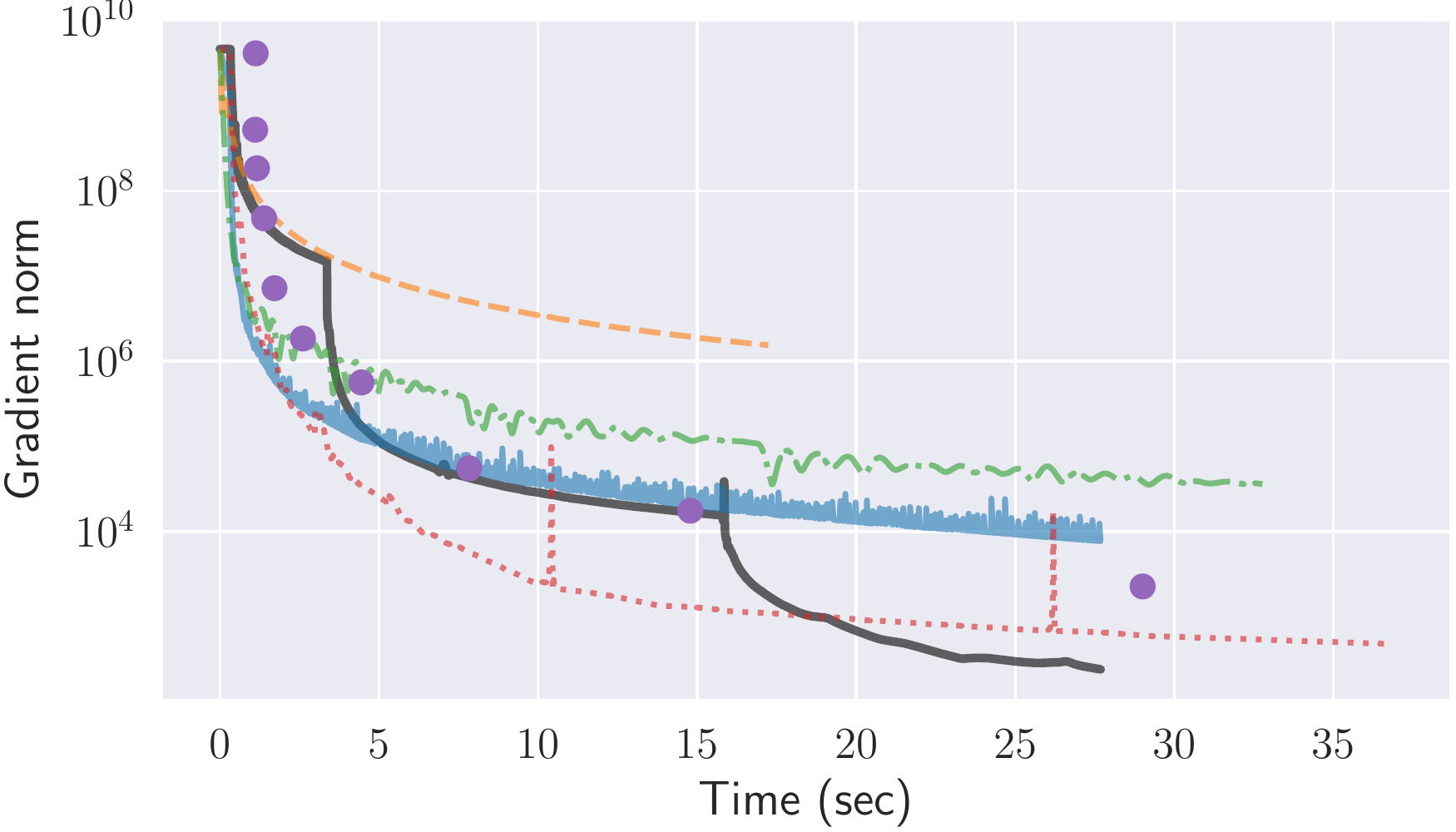}%
  }\hspace{0.1\linewidth}%
  \subfloat[Rosenbrock function \cref{eq:rosenbrock}\label{fig:exp_time_rosenbrock}]{%
    \includegraphics[width=0.43\linewidth]{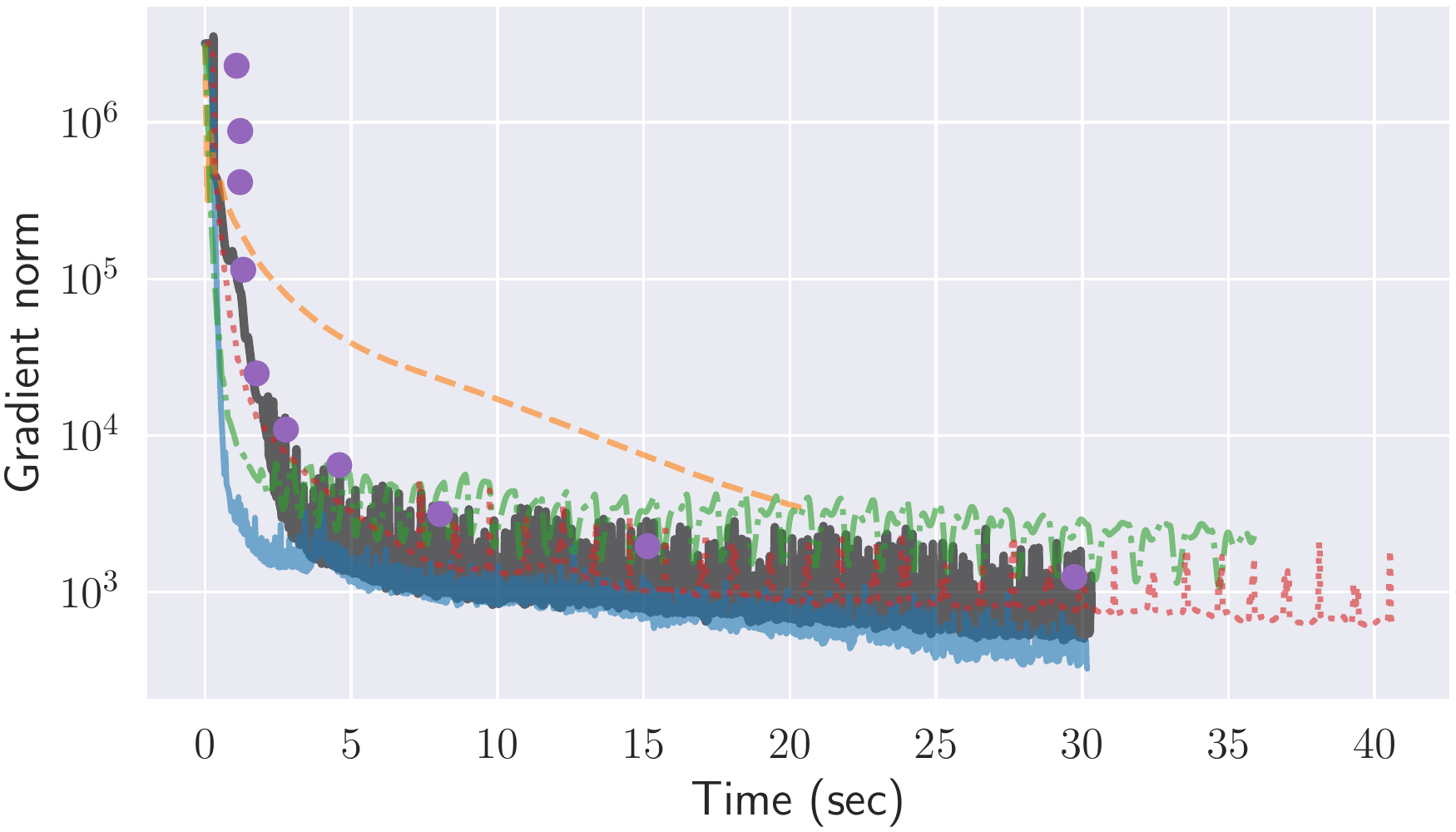}%
  }\par
  \caption{
    \revise{
      Numerical results with benchmark functions.
      The horizontal axis is the elapsed time in seconds.
      \label{fig:experiments_benchmark_time}
    }
  }
\end{figure}

\begin{figure}
  \centering
  \includegraphics[height=3.5ex]{fig/legend.pdf}\par\medskip%
  \subfloat[Classification \cref{eq:exp_classification}\label{fig:exp_classification}]{%
    \includegraphics[width=0.43\linewidth]{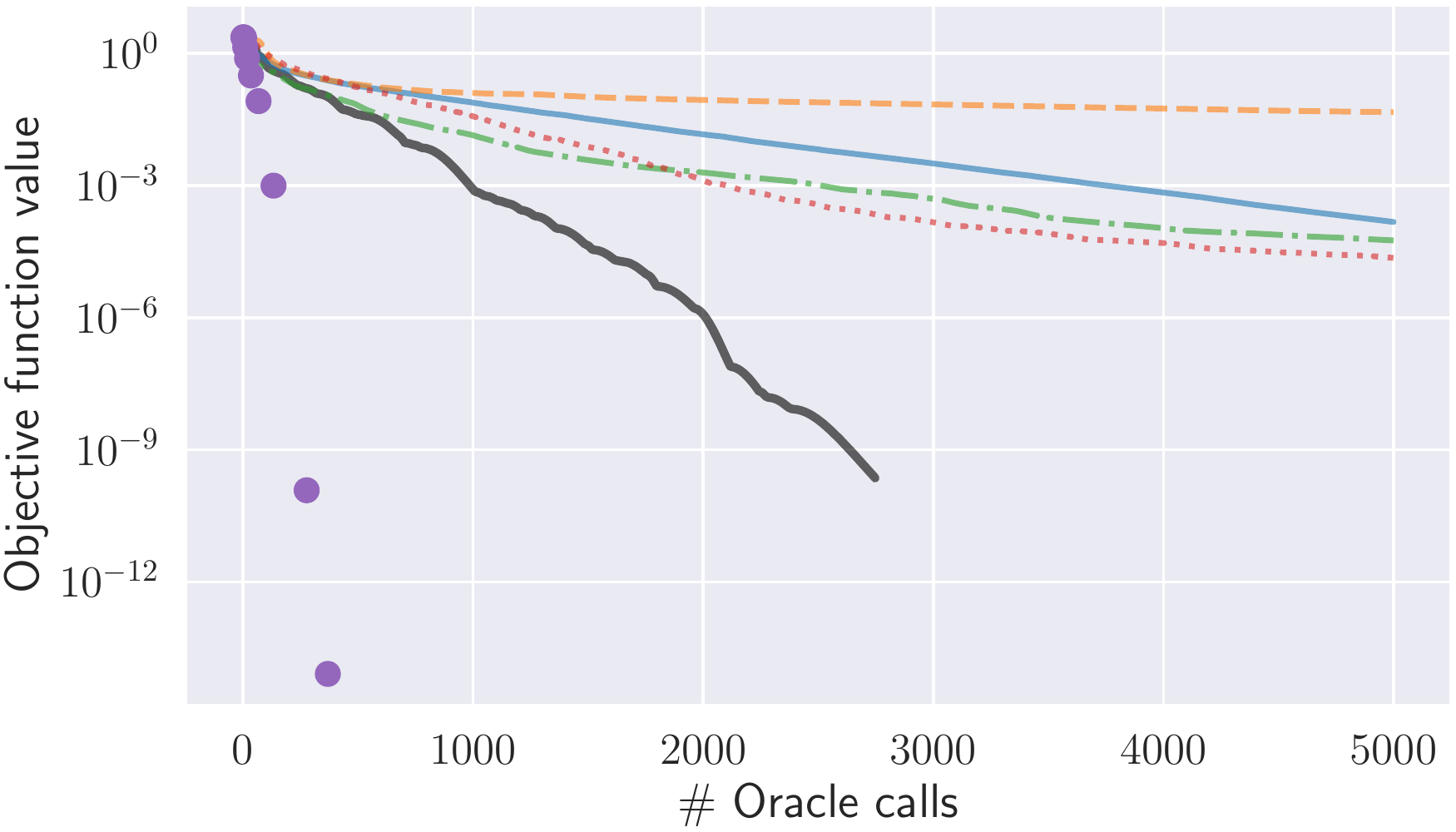}\hspace{0.1\linewidth}%
    \includegraphics[width=0.43\linewidth]{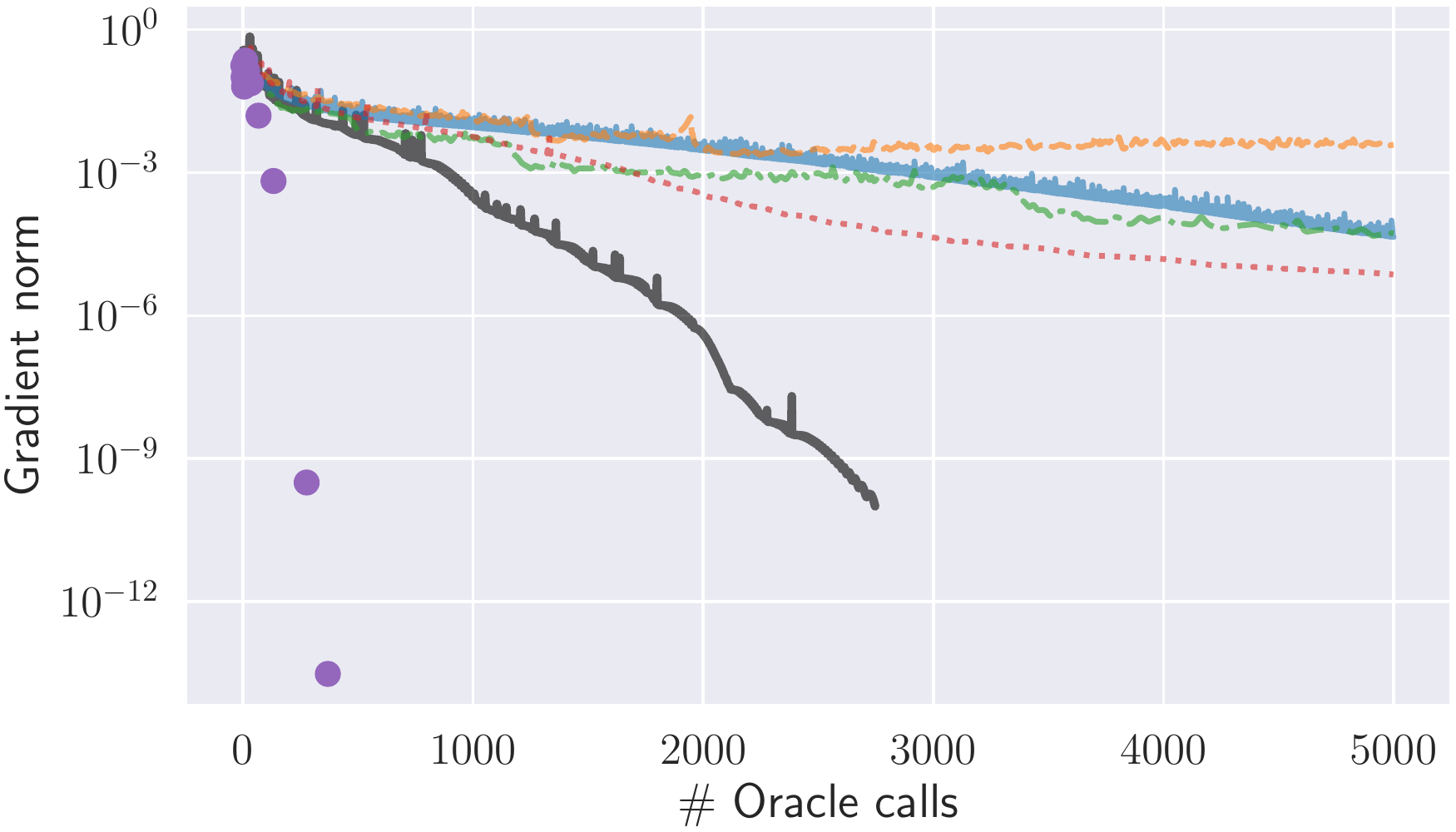}%
  }\par\medskip%
  \subfloat[Autoencoder \cref{eq:exp_ae}\label{fig:exp_autoencoder}]{%
    \includegraphics[width=0.43\linewidth]{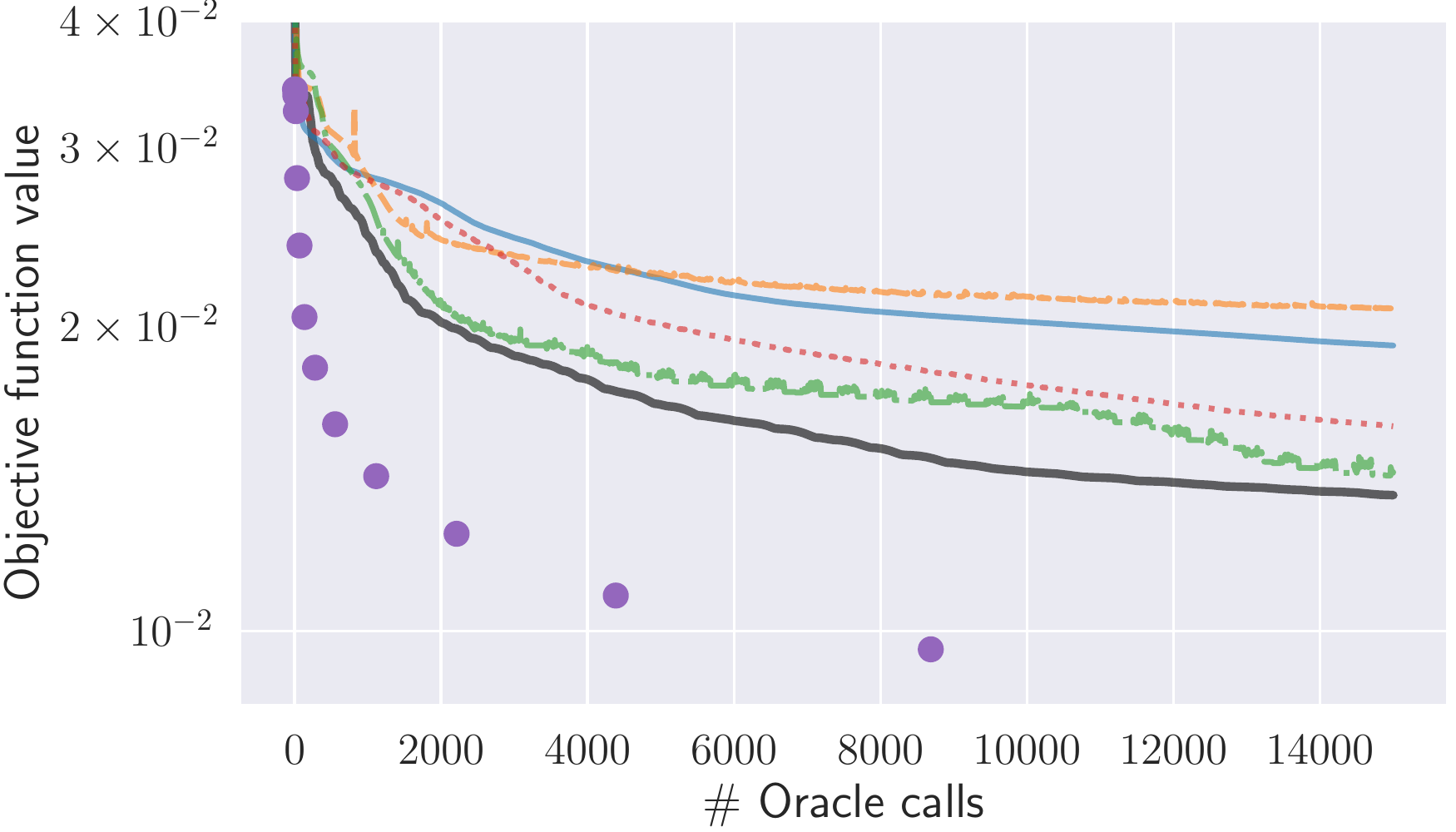}\hspace{0.1\linewidth}%
    \includegraphics[width=0.43\linewidth]{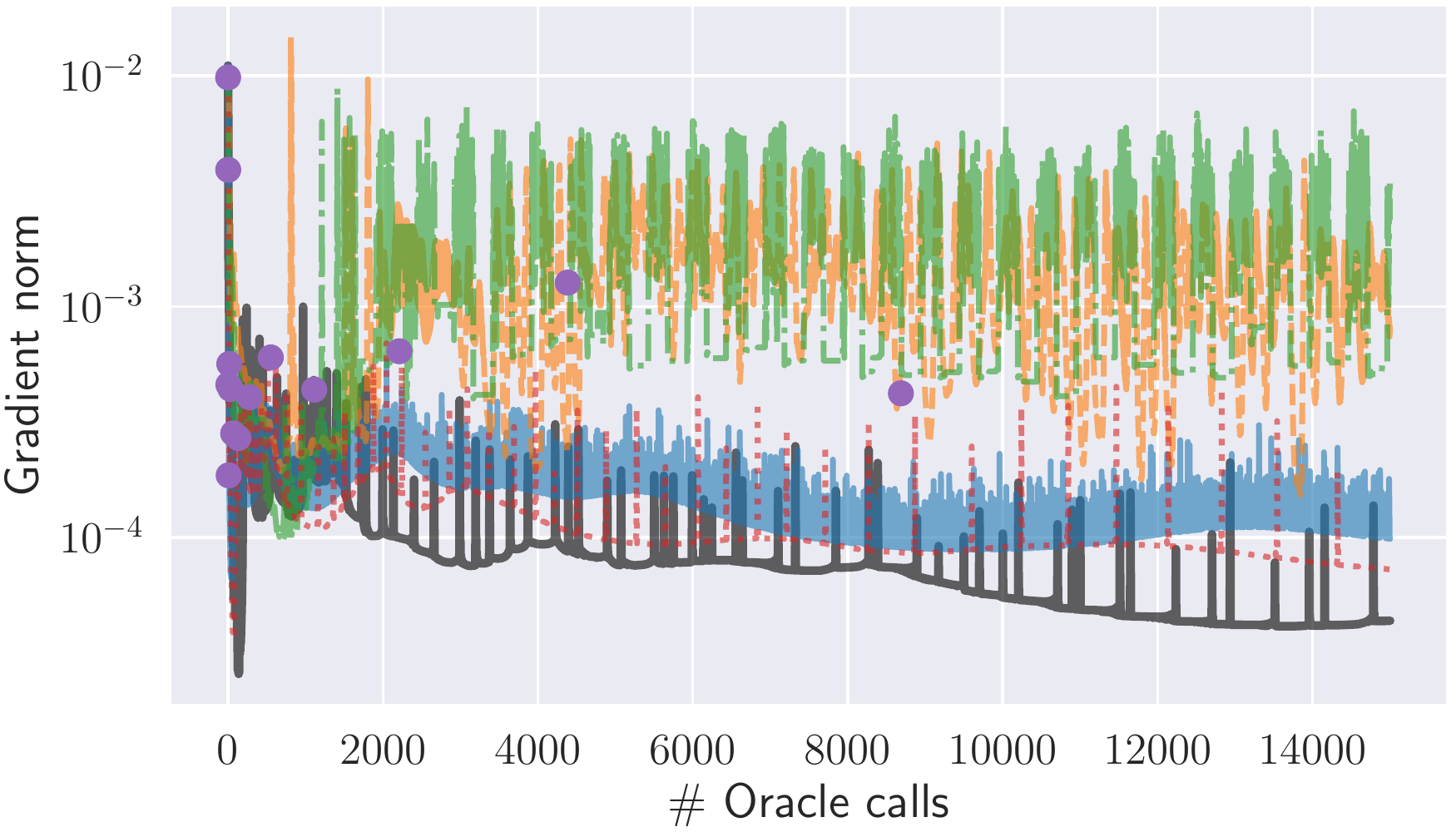}%
  }\par\medskip%
  \subfloat[Matrix completion \cref{eq:exp_mf} ($r = 100$)\label{fig:exp_completion_100}]{%
    \includegraphics[width=0.43\linewidth]{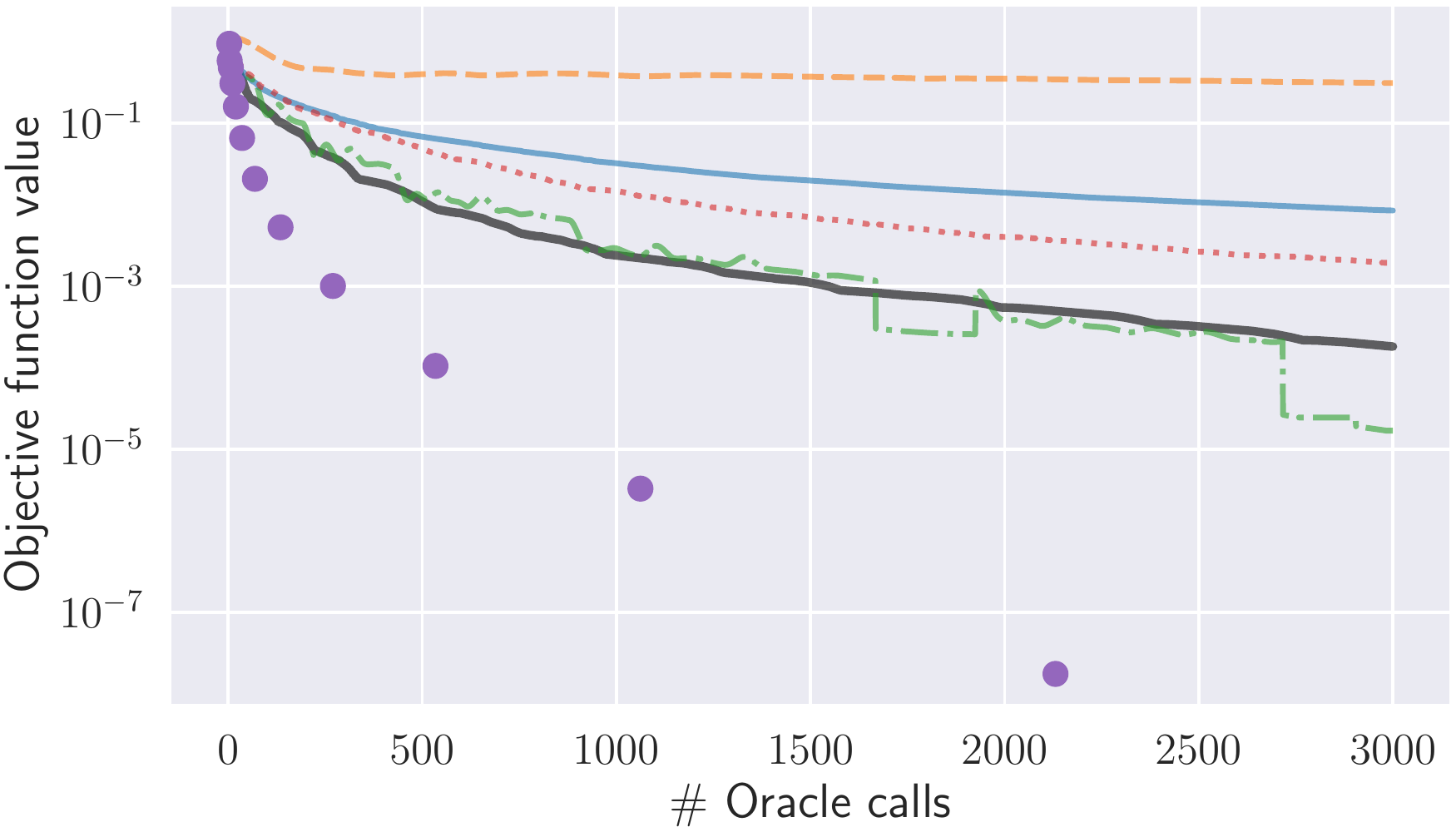}\hspace{0.1\linewidth}%
    \includegraphics[width=0.43\linewidth]{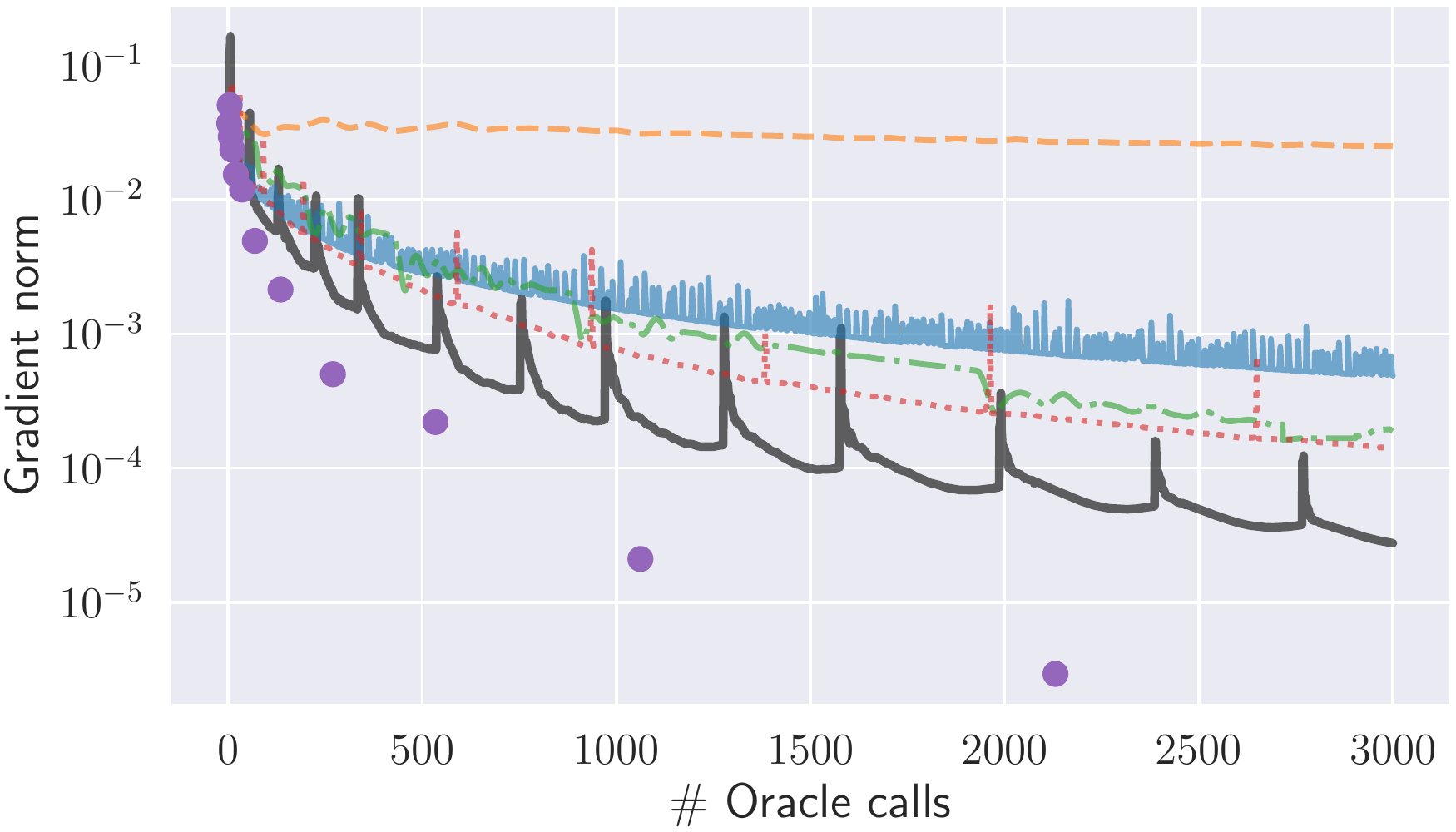}%
  }\par\medskip%
  \subfloat[Matrix completion \cref{eq:exp_mf} ($r = 200$)\label{fig:exp_completion_200}]{%
    \includegraphics[width=0.43\linewidth]{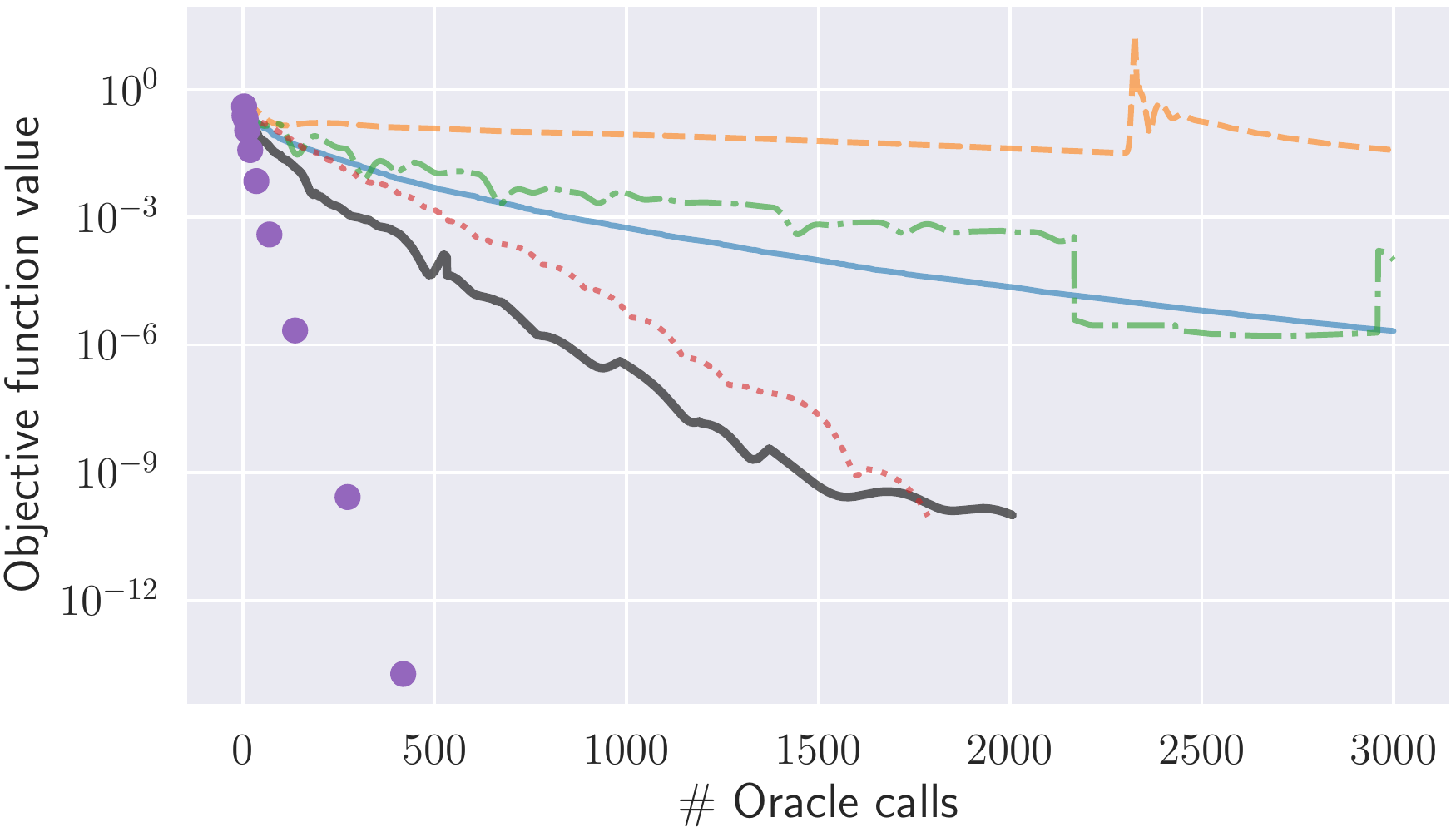}\hspace{0.1\linewidth}%
    \includegraphics[width=0.43\linewidth]{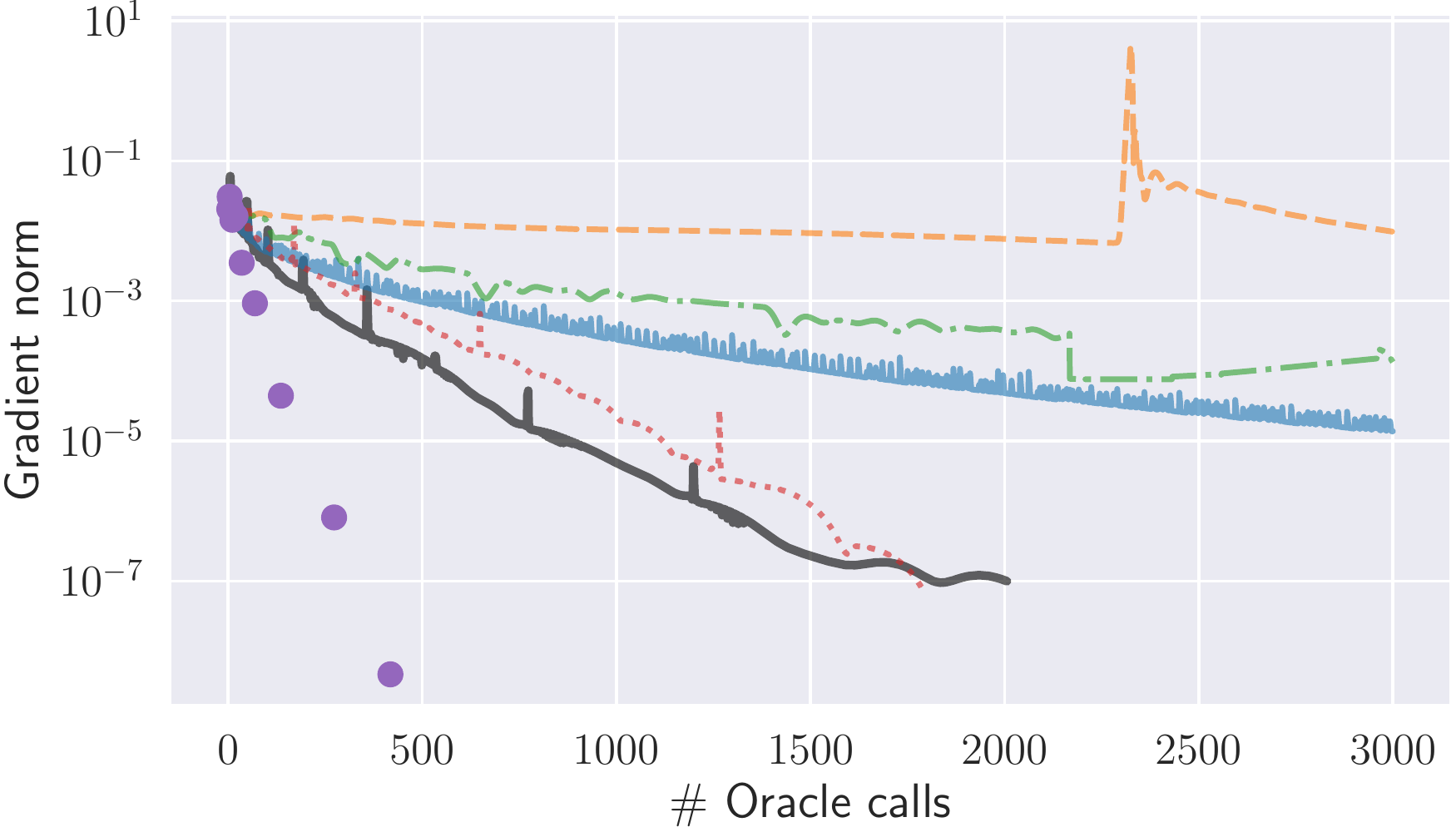}%
  }\par%
  \caption{Numerical results \revise{with machine learning instances.}\label{fig:experiments}}
\end{figure}

\begin{figure}
  \centering
  \includegraphics[height=3.5ex]{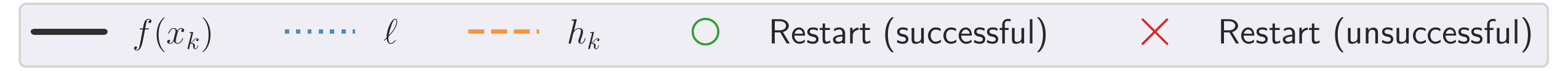}\par\medskip%
  \subfloat[Classification \cref{eq:exp_classification}\label{fig:exp_objlh_classification}]{%
    \includegraphics[width=0.43\linewidth]{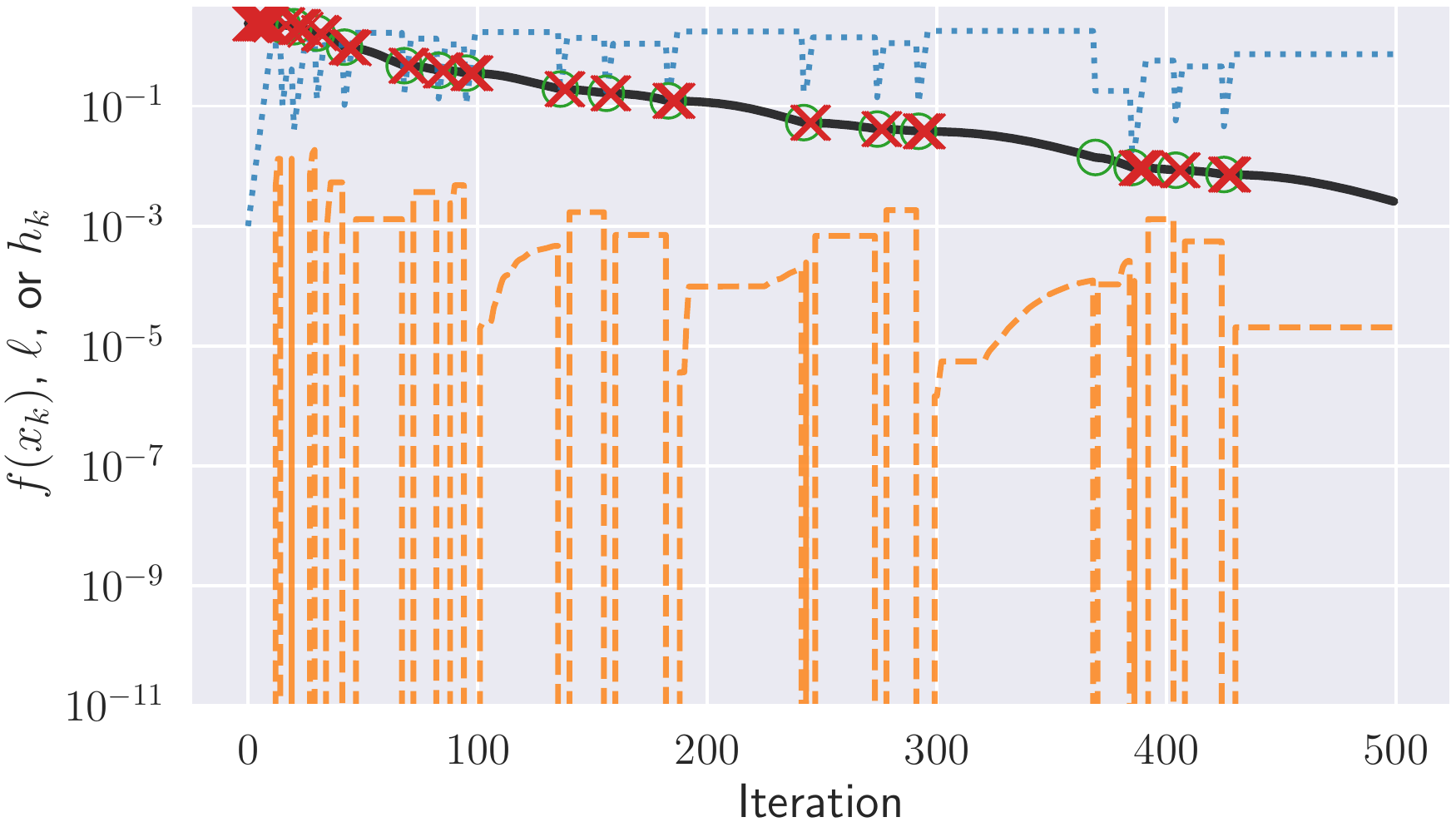}\hspace{0.1\linewidth}%
    \includegraphics[width=0.43\linewidth]{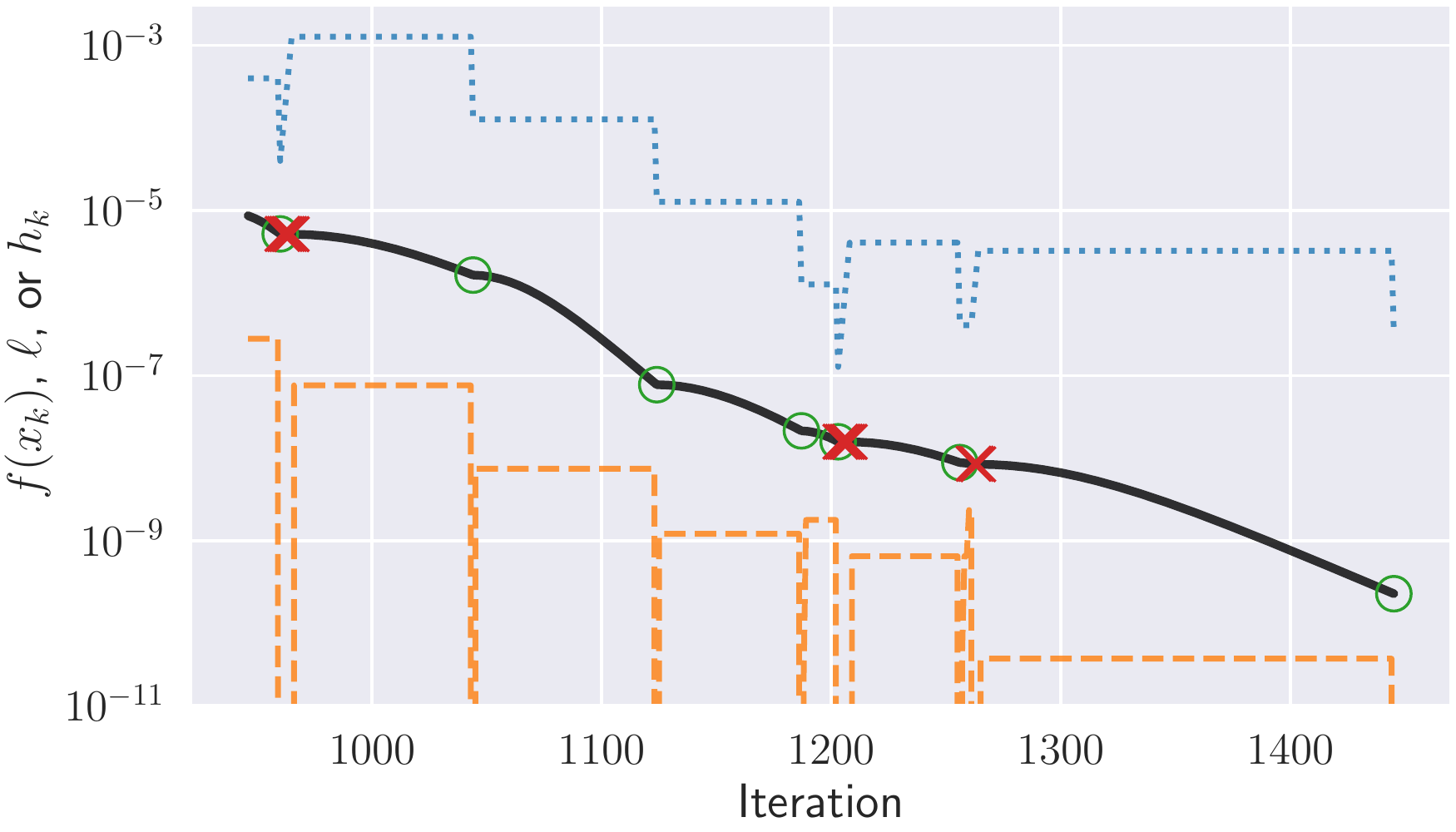}%
  }\par\medskip%
  \subfloat[Autoencoder \cref{eq:exp_ae}\label{fig:exp_objlh_autoencoder}]{%
    \includegraphics[width=0.43\linewidth]{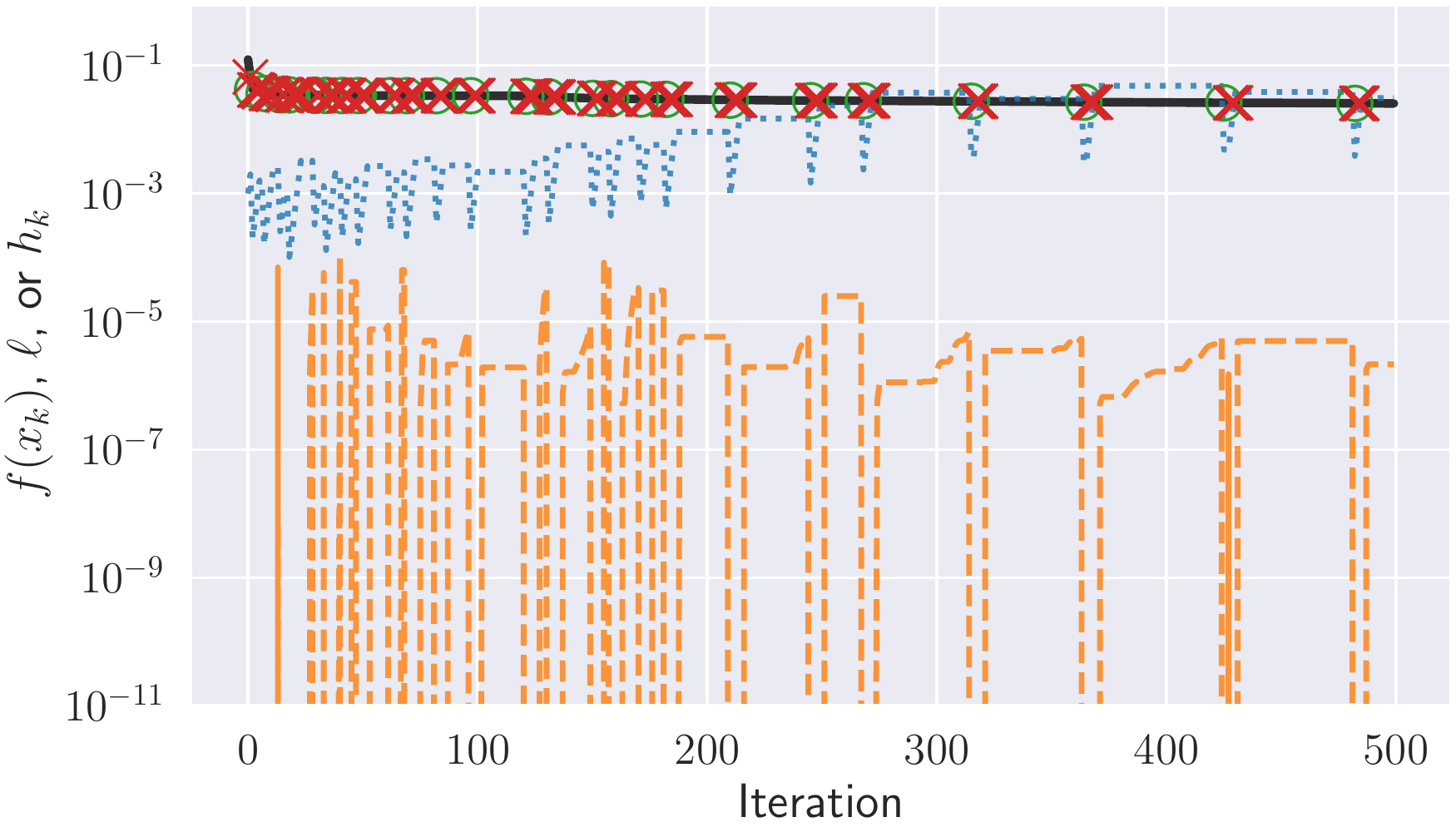}\hspace{0.1\linewidth}%
    \includegraphics[width=0.43\linewidth]{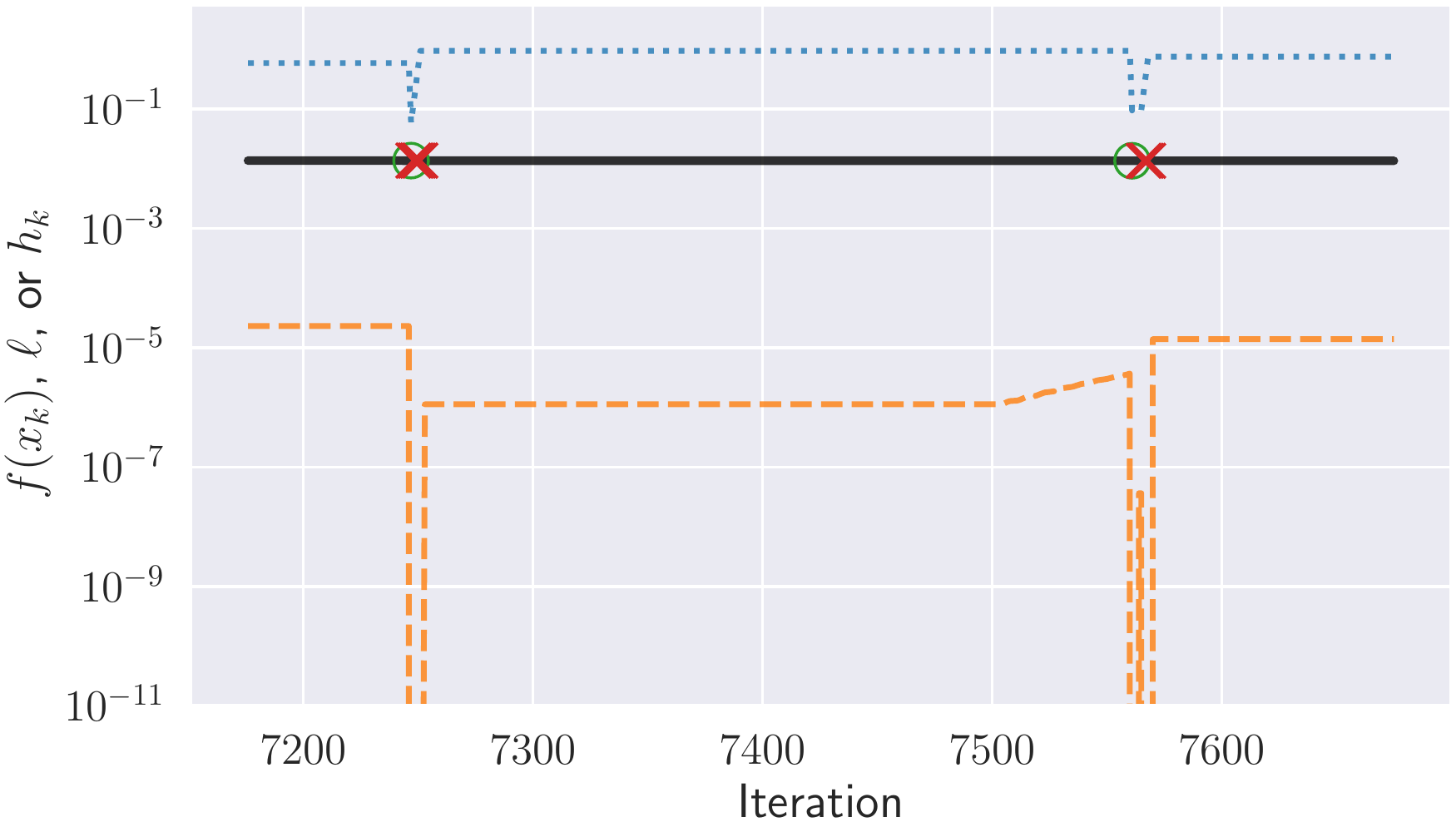}%
  }\par\medskip%
  \subfloat[Matrix completion \cref{eq:exp_mf} ($r = 100$)\label{fig:exp_objlh_completion_100}]{%
    \includegraphics[width=0.43\linewidth]{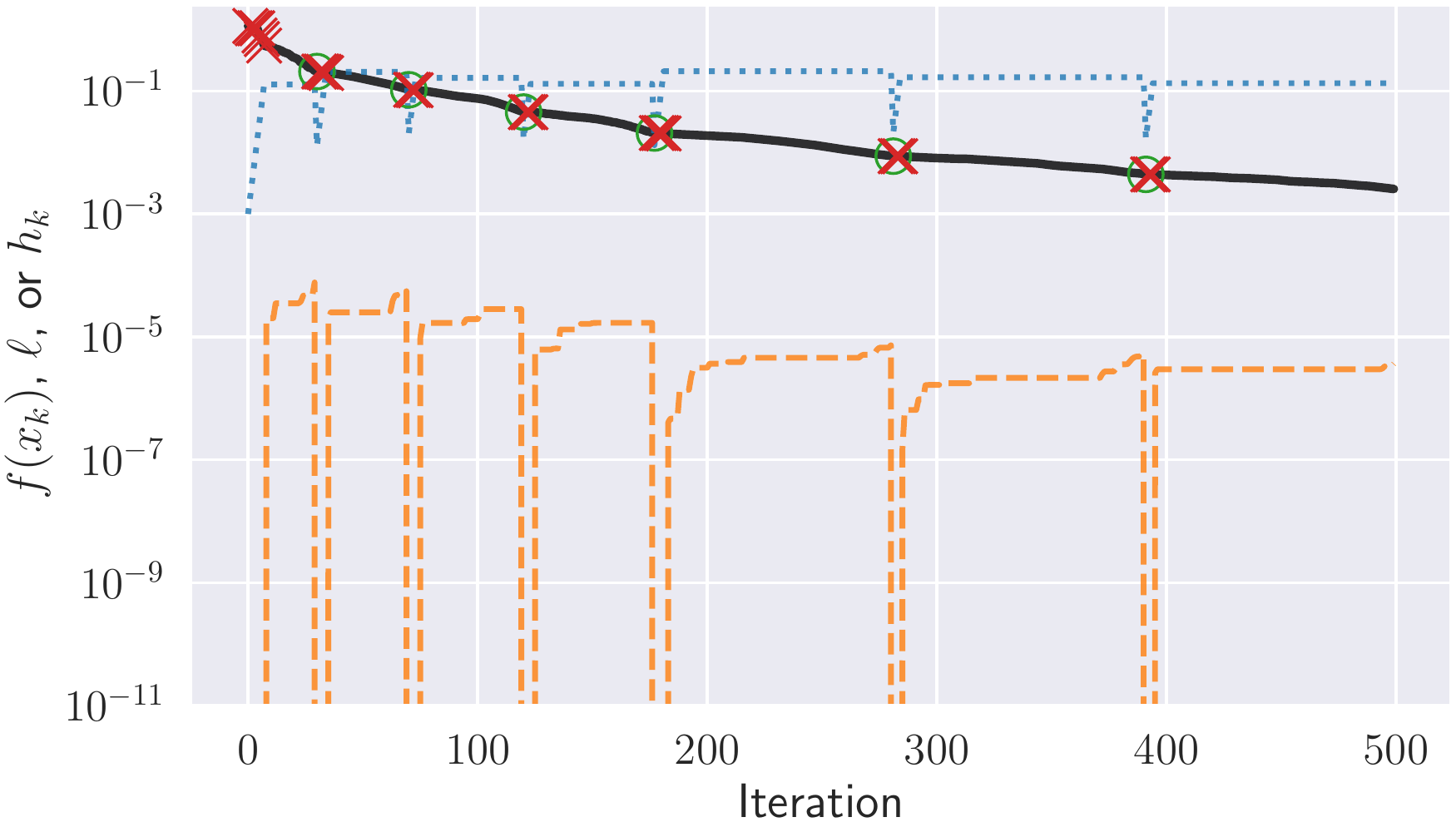}\hspace{0.1\linewidth}%
    \includegraphics[width=0.43\linewidth]{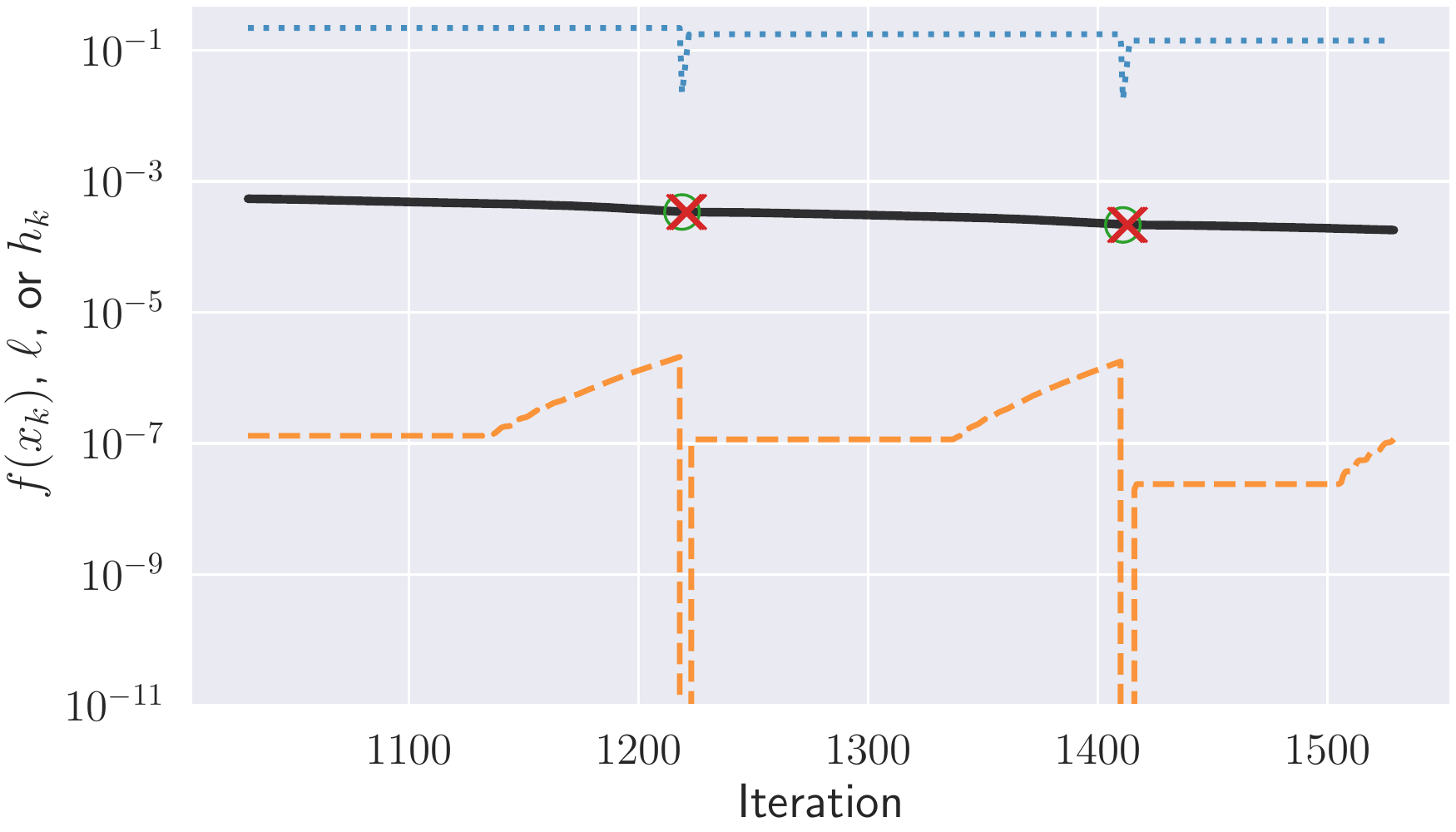}%
  }\par\medskip%
  \subfloat[Matrix completion \cref{eq:exp_mf} ($r = 200$)\label{fig:exp_objlh_completion_200}]{%
    \includegraphics[width=0.43\linewidth]{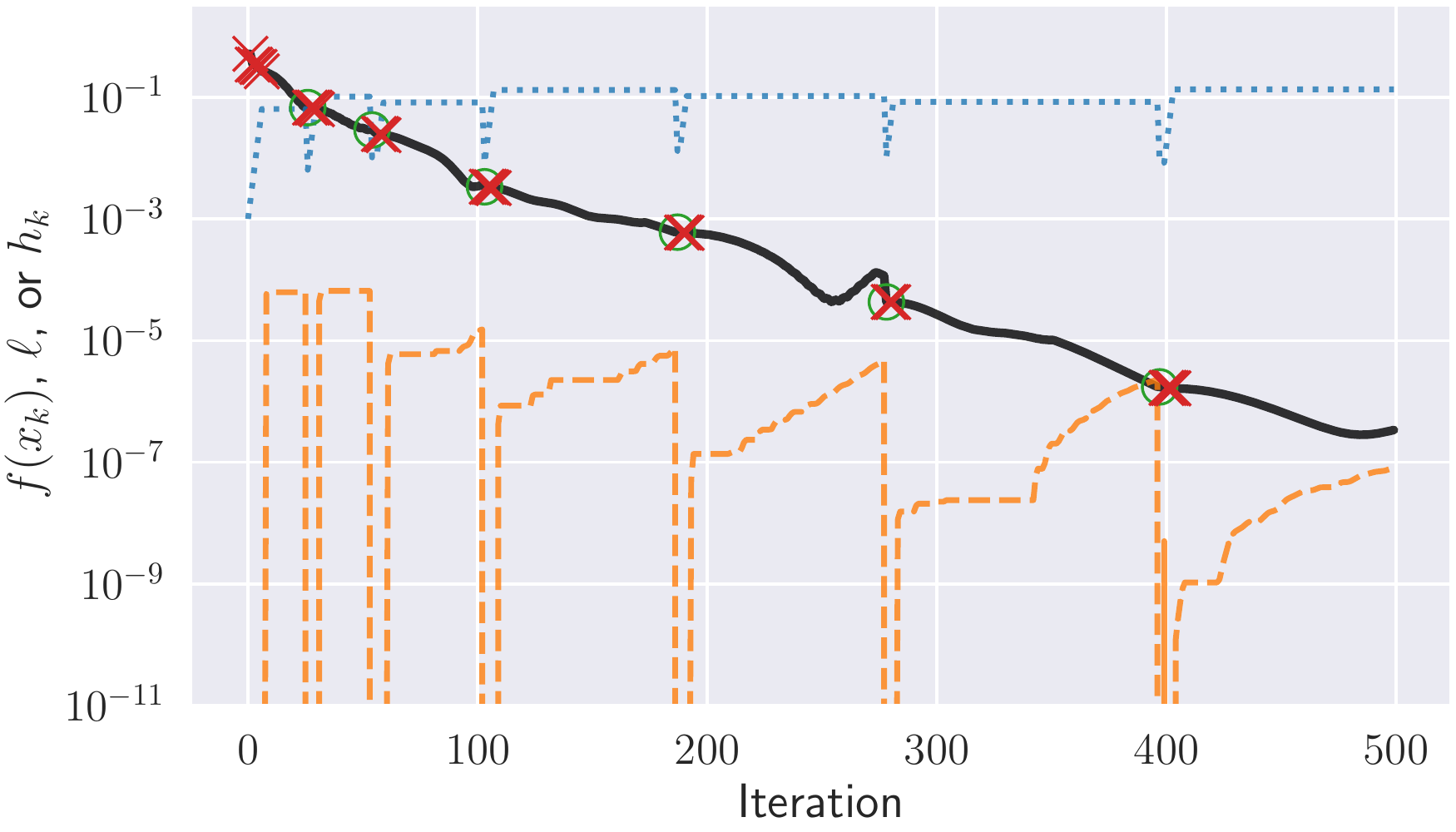}\hspace{0.1\linewidth}%
    \includegraphics[width=0.43\linewidth]{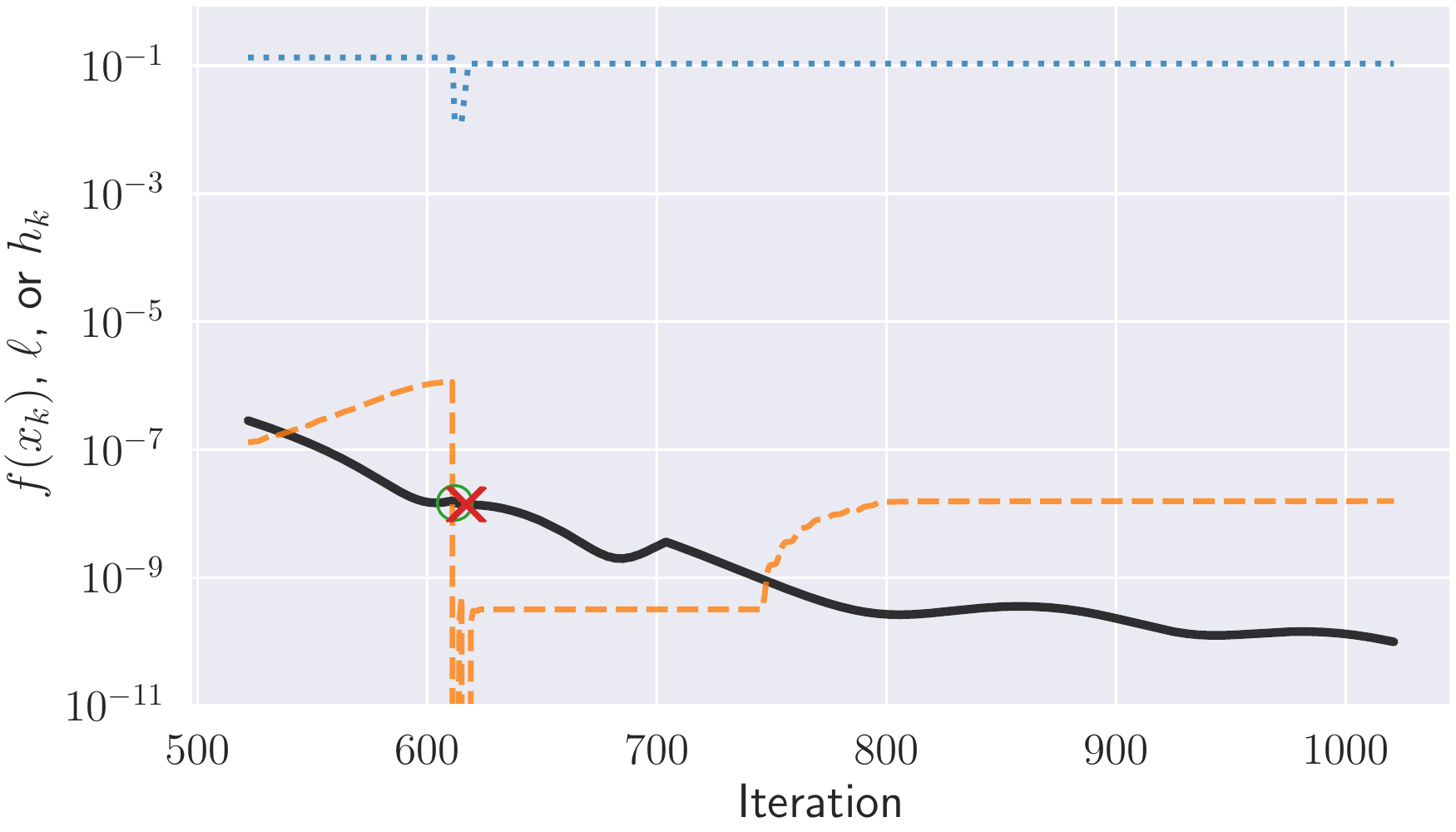}%
  }\par%
  \caption{
    \revise{
      The objective function value $f(x_k)$ and the estimates $\Lest$ and $\Hest_k$ at each iteration of the proposed method.
      The iterations at which a restart occurred are marked.
      Left: the first 500 iterations.
      Right: later 500 iterations.
      \label{fig:experiments_objlh}
    }
  }
\end{figure}



\appendix
\section{Omitted proofs}
\subsection{Proof of \cref{lem:gradient_jensen,lem:trapezoidal_rule_error}}
\label{sec:proof_lem_holder_hessian}
\begin{proof}[Proof of \cref{lem:gradient_jensen}]
  \revise{Since $f$ is twice differentiable, we have}
  \begin{align}
    \nabla f(z_i) - \nabla f(\bar z)
    &=
    \nabla^2 f(\bar z) (z_i - \bar z)
    + \int_0^1 \prn*{ \nabla^2 f(\bar z + t (z_i - \bar z)) - \nabla^2 f(\bar z) } (z_i - \bar z) dt,
  \end{align}
  and its weighted average gives
  \begin{align}
    \sum_{i=1}^n \lambda_i \nabla f(z_i) - \nabla f(\bar z)
    &=
    \sum_{i=1}^n \lambda_i
    \int_0^1 \prn*{ \nabla^2 f(\bar z + t (z_i - \bar z)) - \nabla^2 f(\bar z) } (z_i - \bar z) dt.
  \end{align}
  Therefore, we obtain the first inequality as follows:
  \begin{alignat}{2}
    \norm*{
      \sum_{i=1}^n \lambda_i \nabla f(z_i) - \nabla f(\bar z)
    }
    &\leq
    \sum_{i=1}^n \lambda_i
    \int_0^1 \norm*{ \nabla^2 f(\bar z + t (z_i - \bar z)) - \nabla^2 f(\bar z) } \norm*{z_i - \bar z} dt\\
    &\leq
    \sum_{i=1}^n \lambda_i
    \int_0^1 \Htrue{\nu} \norm*{t(z_i - \bar z)}^\nu \norm*{z_i - \bar z} dt
    &\quad&\by{definition \cref{eq:def_holder_constant}}\\
    &=
    \frac{\Htrue{\nu}}{1 + \nu} \sum_{i=1}^n \lambda_i \norm*{z_i - \bar z}^{1 + \nu}.
  \end{alignat}

  Next, we will prove the second inequality.
  H\"older's inequality gives
  \begin{align}
    \sum_{i=1}^n \lambda_i \norm*{z_i - \bar z}^{1 + \nu}
    &=
    \sum_{i=1}^n
    \lambda_i^{\frac{1 - \nu}{2}}
    \prn*{ \lambda_i \norm*{z_i - \bar z}^2 }^{\frac{1 + \nu}{2}}\\
    &\leq
    \prn*{\sum_{i=1}^n \lambda_i}^{\frac{1 - \nu}{2}}
    \prn*{
      \sum_{i=1}^n \lambda_i \norm*{\revise{z_i} - \revise{\bar z}}^2
    }^{\frac{1 + \nu}{2}}
    =
    \prn*{
      \sum_{i=1}^n \lambda_i \norm*{\revise{z_i} - \revise{\bar z}}^2
    }^{\frac{1 + \nu}{2}}
  \end{align}
  since $\sum_{i=1}^n \lambda_i = 1$.
  Furthermore, \revise{we have} $\sum_{i=1}^n \lambda_i \norm*{z_i - \bar z}^2 = \sum_{1 \leq i < j \leq n} \lambda_i \lambda_j \allowbreak \|z_i - z_j\|^2$ \revise{because}
  \begin{align}
    \sum_{1 \leq i < j \leq n} \lambda_i \lambda_j \norm*{z_i - z_j}^2
    &= 
    \frac{1}{2} \sum_{i,j = 1}^n \lambda_i \lambda_j \norm*{z_i - z_j}^2\\
    &= 
    \frac{1}{2} \sum_{i,j = 1}^n \lambda_i \lambda_j \norm*{z_i}^2
    + \frac{1}{2} \sum_{i,j = 1}^n \lambda_i \lambda_j \norm*{z_j}^2
    - \sum_{i,j = 1}^n \inner{\lambda_i z_i}{\lambda_j z_j}\\
    &= 
    \sum_{i=1}^n \lambda_i \norm*{z_i}^2 - \norm*{\bar z}^2
    =
    \sum_{i=1}^n \lambda_i \norm*{z_i - \bar z}^2,
  \end{align}
  which completes the proof.
\end{proof}

\begin{proof}[Proof of \cref{lem:trapezoidal_rule_error}]
  We obtain the desired result as follows:
  \begin{alignat}{2}
    &\mathInd
      f(x) - f(y)
      - \frac{1}{2} \inner*{\nabla f(x) + \nabla f(y)}{x - y}\\
    &=
      \int_0^1 \inner*{\nabla f(t x + (1 - t) y)}{x - y} dt
      - \frac{1}{2} \inner*{\nabla f(x) + \nabla f(y)}{x - y}
    \\
    &=
      \int_0^1 \inner[\Big]{\nabla f(t x + (1 - t) y) - t \nabla f(x) - (1 - t) \nabla f(y)}{x - y} dt\\
    &\leq
    \int_0^1 \norm[\Big]{\nabla f(t x + (1 - t) y) - t \nabla f(x) - (1 - t) \nabla f(y)} \norm*{x - y} dt\\
    &\leq
    \frac{\Htrue{\nu}}{1 + \nu}
    \int_0^1
    \prn*{
      t (1 - t)^{1 + \nu}
      + (1 - t) t^{1 + \nu}
    }
    \norm*{x - y}^{2 + \nu} dt
    &\quad&\by{\cref{lem:gradient_jensen}}\\
    &=
    \frac{2 \Htrue{\nu}}{(1 + \nu) (2 + \nu) (3 + \nu)} \norm*{x - y}^{2 + \nu}.
  \end{alignat}
  \revise{
    For the last inequality, we used \cref{lem:gradient_jensen} with $n = 2$, $z_1 = x$, $z_2 = y$, $\lambda_1 = t$, and $\lambda_2 = 1-t$, obtaining
    \begin{align}
      &\mathInd
      \norm*{
        \nabla f(t x + (1 - t) y) - t \nabla f(x) - (1 - t) \nabla f(y)
      }\\
      &\leq
      \frac{\Htrue{\nu}}{1 + \nu}
      \prn*{
        t \norm*{x - (t x + (1 - t) y)}^{1 + \nu}
        + (1-t) \norm*{y - (t x + (1 - t) y)}^{1 + \nu}
      }\\
      &=
      \frac{\Htrue{\nu}}{1 + \nu}
      \prn*{
        t (1 - t)^{1 + \nu}
        + (1 - t) t^{1 + \nu}
      }
      \norm*{x - y}^{1 + \nu}.
    \end{align}
  }
\end{proof}

\subsection{Proof of \cref{eq:grad_norm_xbar_upperbound}}
\label{sec:proof_eq_grad_norm_xbar_upperbound}
Inequality~\cref{eq:grad_norm_xbar_upperbound} is a modification of \citep[Eq.~(22)]{marumo2022parameter}, which was originally for an accelerated gradient method with Lipschitz continuous Hessians, for our heavy-ball method with H\"older continuous Hessians.
The following proof of \cref{eq:grad_norm_xbar_upperbound} is based on the one for \citep[Eq.~(22)]{marumo2022parameter} but is easier, thanks to our simple choice of $\theta_k = 1$.
\begin{proof}
  \revise{Using the triangle inequality and \cref{lem:gradient_jensen} with $n = k$, $z_i = x_i$, and $\lambda_i = \frac{1}{k}$ yields}
  \begin{alignat}{2}
    \norm*{
      \nabla f (\bar x_k)
    }
    &
    \reviset{
    {}\leq
    \frac{1}{k}
    \norm*{
      \sum_{i=0}^{k-1} \nabla f (x_i)
    }
    +
    \norm*{
      \nabla f (\bar x_k) - \frac{1}{k} \sum_{i=0}^{k-1} \nabla f (x_i)
    }
    }
    \\
    &\leq
    \frac{1}{k}
    \norm*{
      \sum_{i=0}^{k-1} \nabla f (x_i)
    }
    +
    \frac{\Htrue{\nu}}{1 + \nu}
    \prn[\Bigg]{
      \frac{1}{k^2} \sum_{0 \leq i < j < k} \norm*{x_i - x_j}^2
    }^{\frac{1 + \nu}{2}},
    \label{eq:grad_at_average_upperbound}
  \end{alignat}
  and we will evaluate each term.
  First, it follows from the update rule~\cref{eq:update_v} that
  \begin{align}
    \sum_{i=0}^{k-1}
    \nabla f(x_i)
    &=
    \Lest
    \sum_{i=0}^{k-1} (v_i - v_{i+1})
    =
    \Lest (v_0 - v_k)
    =
    - \Lest v_k.
  \end{align}
  Therefore, the first term on the right-hand side of \cref{eq:grad_at_average_upperbound} reduces to $\frac{\Lest}{k} \norm*{v_k}$.
  Next, we bound the second term.
  Using the triangle inequality and the Cauchy--Schwarz inequality yields
  \begin{align}
    \sum_{0 \leq i < j < k} \norm*{x_i - x_j}^2
    \leq
    \sum_{0 \leq i < j < k}
    \prn*{
      \sum_{l=i+1}^j \norm*{v_l}
    }^2
    &\leq
    \sum_{0 \leq i < j < k}
    \prn*{\sum_{l=i+1}^j 1^2}
    \prn*{\sum_{l=i+1}^j \norm*{v_l}^2}\\
    &=
    \sum_{0 \leq i < j < k}
    (j - i)
    \sum_{l=i+1}^j \norm*{v_l}^2,
  \end{align}
  and interchanging the summations leads to
  \begin{align}
    \revise{
    =
    \sum_{0 \leq i < l \leq j < k}
    (j - i) \norm*{v_l}^2
    }
    &=
    \sum_{l=1}^{k-1}
    \prn[\Bigg]{
      \sum_{i=0}^{l-1}
      \sum_{j=l}^{k-1}
      (j - i)
    }
    \norm*{v_l}^2\\
    &=
    \frac{k}{2}
    \sum_{l=1}^{k-1}
    l (k - l)
    \norm*{v_l}^2
    \leq
    \frac{k}{2}
    \sum_{l=1}^{k-1} \frac{k^2}{4} \norm*{v_l}^2
    \leq
    \frac{k^3}{8} S_k.
  \end{align}
  We obtain the desired result by evaluating the right-hand side of \cref{eq:grad_at_average_upperbound}.
\end{proof}

\section*{Acknowledgments}
This work was partially supported by JSPS KAKENHI (19H04069) and JST ERATO (JPMJER1903).

\section*{Conflict of interest}
The authors declare that they have no conflict of interest.

\bibliographystyle{abbrvnat}
\bibliography{../../../../../../../tex/myrefs}

\end{document}